\documentclass[a4paper,12pt]{article}
\pdfoutput=1
\usepackage{mathtools} 
\usepackage{extarrows}
\usepackage{amsmath}
\usepackage{amsthm}
\usepackage{graphicx}
\usepackage{verbatim}
\usepackage{epsfig}
\usepackage{hyperref}
\usepackage{float}
\usepackage{color}
\usepackage{fullpage}
\usepackage{amsfonts}
\usepackage{amscd}
\usepackage{amssymb}
\usepackage{graphics}
\usepackage{amsmath}
\usepackage[USenglish]{babel}
\usepackage{bbm}
\usepackage{yfonts}
\usepackage{mathrsfs}
\usepackage{color}
\usepackage{caption}
\usepackage{subcaption}
\usepackage{graphicx}
\usepackage[titletoc,toc,title]{appendix}
\DeclareFontFamily{OML}{rsfs}{\skewchar\font'177}
\DeclareFontShape{OML}{rsfs}{m}{n}{ <5> <6> rsfs5 <7> <8> <9>
rsfs7 <10> <10.95> <12> <14.4> <17.28> <20.74> <24.88> rsfs10 }{}
\DeclareMathAlphabet{\mathfs}{OML}{rsfs}{m}{n}

\newcommand{\BE}{{\mathbb{E}}}

\newcommand{\BH}{{\mathbb{H}}}

\newcommand{\BL}{{\mathbb{L}}}

\newcommand{\BP}{{\mathbb{P}}}

\newcommand{\BR}{{\mathbb{R}}}
\newcommand{\BS}{{\mathbb{S}}}
\newcommand{\BT}{{\mathbb{T}}}

\newcommand{\BZ}{{\mathbb{Z}}}

\newcommand{\CC}{{\mathcal{C}}}

\newcommand{\CH}{{\mathcal{H}}}

\newcommand{\CT}{{\mathcal{T}}}

\usepackage{enumitem}
\makeatletter
\def\namedlabel#1#2{\begingroup
    #2%
    \def\@currentlabel{#2}%
    \phantomsection\label{#1}\endgroup
}
\makeatother

\newcommand{\Cyl}{{\rm Cyl}}
\renewcommand{\d}{{\rm d}}

\newcommand{\dist}{{\rm dist}}

\newcommand{\Leb}{{\rm Leb}}
\newcommand{\Vol}{{\rm Vol}}

\newtheorem{theorem}{Theorem}[section]

\newtheorem{lemma}[theorem]{Lemma}

\begin{document}
\numberwithin{equation}{section} \numberwithin{figure}{section}
\title{Higher-dimensional stick percolation}
\author{Erik I. Broman }
\maketitle
\begin{abstract}
We consider two cases of the so-called stick percolation model
with sticks of length $L.$
In the first case, the orientation is chosen independently 
and uniformly, while in the second all sticks are oriented along the 
same direction.  
We study their respective critical values $\lambda_c(L)$ of the 
percolation phase transition, and in particular we investigate 
the asymptotic behavior of $\lambda_c(L)$ as $L\to \infty$ for 
both of these cases. In the first case we prove that 
$\lambda_c(L)\sim L^{-2}$ for any $d\geq 2,$ while in the second 
we prove that $\lambda_c(L)\sim L^{-1}$ for any $d\geq 2.$ 
 
\end{abstract}


\section{Introduction}

The two-dimensional Poisson stick model is a classical 
continuum-percolation model. The first paper focused on this model was 
\cite{R_91}, although earlier papers such as \cite{DK_84} and \cite{H_85}
included this model in their general framework.
As explained in \cite{MR_96}, the model was motivated by material 
sciences. For instance, it can be used to model the effect 
of fractures in a material to the overall strength and brittleness of 
said material, or when studying fault lines in geological structures. 

Since the introduction of this model, a host of papers on the subject 
has appeared in the physics literature (see for example \cite{MNT_14}
and \cite{TE_19} and the references therein). Many of these deal with a 
three-dimensional variant where the two-dimensional stick is replaced 
by other percolation objects, such as nanotubes or nanowires, which 
are suspended in some other material. These models have been used to
model various phenomena such as thin film transistors, flexible
microelectronics, microelectromechanical systems, chemical sensors 
and the construction 
of transparent electrodes for optoelectronic and photovoltaic devices
(as stated in \cite{MNT_14}). In order to have a concrete example in mind, 
we can consider 
a material consisting of conductive nanowires suspended
in a non-conductive substance. The material is then said to be 
conductive if large connected components of nanowires exist
through which current can flow. 

The main purpose and motivation of this paper is to perform a 
rigorous mathematical 
analysis of such models in higher-dimensional space. The physically 
most relevant examples are clearly when $d=2$ (e.g.~thin films) or $d=3$
(e.g.~conductivity of suspended nanowires), but from a 
mathematical viewpoint it is desirable to obtain a general 
result that works for all $d\geq 2.$

A secondary motivation comes from the recent study 
(see \cite{TW_12} and \cite{BT_16}) of the so-called Poisson 
cylinder model. 
This is a model where the percolation objects are infinitely long,
and while it was shown in \cite{TW_12} that the vacant set undergoes 
a phase transition, it was later shown in \cite{BT_16} that the occupied 
component does not. (We point out that, since the cylinders are unbounded,
the phase transition in question is not whether there exists unbounded
connected components, but rather whether {\em all} cylinders in the
model belong to the same connected set.) 
It is natural to think of the Poisson cylinder model as a limit of the 
Poisson stick model as the length $L$ of the sticks diverges. 
From this viewpoint it then becomes natural to ask 
how the percolation threshold behaves in this limit. 

The main model (see Section \ref{sec:stickmodel} for a 
precise definition) studied in this paper can informally be described 
as follows. Start with a homogeneous Poisson point process in 
$\BR^d$ where $d\geq 2.$ For every point $x$ belonging to this process,
we place a stick centered at $x$ and we let this stick be 
of length $L$ and radius $1.$ Then, we let the orientation of the 
sticks be chosen independently and according to some distribution.
In this paper we will mainly focus on two 
cases. Firstly, the {\em uniform} case in which the orientation
distribution is uniform, and secondly, the {\em rigid} 
case where the orientation is always along the same direction.
The reason for considering these two cases is that they represent two 
extremes among all possible choices of distributions. 
We point out that, while these two cases are our main focus, our 
results hold for a more general setting 
(see Theorem \ref{thm:uniform2}).

Let $\lambda$ denote the intensity of the Poisson point process. 
As usual, we say that percolation occurs if the sticks form
at least one unbounded connected component. Furthermore, we let
$\lambda_{c,u}=\lambda_{c,u}(L)$ denote the critical value of the 
percolation  phase transition (see Section 
\ref{sec:stickmodel} for a precise definition) in the uniform case, 
and we let $\lambda_{c,r}=\lambda_{c,r}(L)$ denote the corresponding 
critical value in the rigid case. The focus of this paper is 
the behavior of $\lambda_{c,u}(L)$ and $\lambda_{c,r}(L)$ as 
$L\to \infty.$ As a motivation we informally 
consider a variant of the uniform case when $d=2.$ Consider 
in this case a model with sticks 
of length $1$ and width $0,$ while the intensity is $\lambda.$  Then, we
increase the length to $L>1,$ after which we rescale space by 
a factor of $L^{-1}$ to recover sticks of length $1.$ After these 
two steps we are back with the original model, but with a new intensity 
which is $L^2\lambda.$ From this, one sees that the critical 
value for this model must scale like $L^{-2}.$ Of course, this 
scaling is exact since the width of the sticks is 0 rather than 1, 
but morally we
should then have that also $\lambda_{c,u}(L)$ scales like $L^{-2}$
when $d=2.$ The reason for this is that when 
$d=2,$ in order for two sticks to overlap, the center lines of the 
sticks tend to cross each other, rather than the sticks just touching.
Thus in this case, the width plays a minor role.
However, when $d\geq 3,$ no center lines will cross and so the width
is crucial. We can
therefore not generalize the intuition from $d=2,$ and it
is natural to investigate what happens when $d\geq 3.$

There are two main results of this paper corresponding to the two above 
mentioned cases. The uniform case will by far require the most effort
and so we present this result first.

\begin{theorem} \label{thm:uniform}
For any dimension $d\geq 2,$ there exist constants 
$0<c<C<\infty,$ only depending on the dimension, 
such that for every $L$ large enough,
\[
cL^{-2}\leq \lambda_{c,u}(L) \leq CL^{-2}.
\]
\end{theorem}
\noindent
{\bf Remarks:} Let us informally turn to the applications of conductivity 
of nanowires suspended in some other substance. Then, the implication 
of Theorem \ref{thm:uniform} is that a doubling of the length of these 
wires implies that roughly a quarter of the original number of sticks will 
suffice in order to maintain connectivity. Theorem \ref{thm:uniform} shows 
that this result does not depend on $d$ other than (possibly) through 
the values of the constants.

We do not have an easy intuitive argument for why the scaling should be the
same in all dimensions. However, it is possible to provide a 
back-of-the-envelope calculation indicating that the scaling should
indeed be $L^{-2}.$ This is done in Section \ref{sec:altstates}.

We prove a slightly stronger version of Theorem \ref{thm:uniform},
(i.e.~Theorem \ref{thm:uniform2}) in which we provide explicit 
bounds on the constants $c$ and $C$ of Theorem \ref{thm:uniform}. 
Furthermore, this theorem also provides a lower 
bound on how large $L$ must be in order for the result to hold.
However, since both the values of the constants and the bound for 
$L$ are presumably far from optimal, we chose 
not to state them here. 

As will be clear from the proofs, the lower bound of Theorem 
\ref{thm:uniform} is in fact universal for every orientation distribution
of the sticks. Furthermore, Theorem \ref{thm:uniform2} provides a result 
which holds for every orientation distribution where the density is
uniformly bounded from below. 

\medskip

It is natural to suspect that $L^{-2}$ is not the correct scaling
for every orientation distribution, and indeed, the rigid case
exhibits a different scaling as stated in our second main result.

\begin{theorem} \label{thm:rigid}
For any dimension $d\geq 2,$ there exist constants 
$0<c<C<\infty,$ only depending on the dimension, 
such that for every $L$ large enough we have that 
\[
c L^{-1}\leq \lambda_{c,r}(L) \leq CL^{-1}.
\]
\end{theorem}
{\bf Remark:} As we will see, the proof of Theorem \ref{thm:rigid}
will be easier than the proof of Theorem \ref{thm:uniform}. Again, we 
will in be able to provide explicit bounds on the constants 
$c$ and $C$ (see Theorem \ref{thm:rigid2}), but as with Theorem 
\ref{thm:uniform}, these are presumably far from optimal. In addition 
we will provide a lower bound on $L$ such that the statement holds. 
\medskip

One may ask 
whether any of the two cases considered here are physically realistic, 
and one might
desire a more complete picture. Ideally, one would like to have a 
result where the 
scaling can be expressed as a function of the orientation distribution 
for any such distribution. At this point we do not see a way to 
prove such a result. However, we do believe that the techniques 
used in this paper can be used to study also other (special) cases of 
orientation distributions, but
that including more cases would obfuscate 
the clarity of the arguments and would unnecessarily lengthen 
the paper.

Even though the results of Theorems \ref{thm:uniform} and \ref{thm:rigid}
differ in the sense that they establish different scaling behavior, 
they in fact rely on the same proof techniques. The proofs of the lower 
bounds (i.e.~$\lambda_{c,u}\geq c L^{-2}$  and $\lambda_{c,r}\geq c L^{-1}$)
will be completed by coupling the stick percolation model with a 
subcritical Galton-Watson process (see the beginning of 
Section \ref{sec:lowerbounds} for a somewhat longer heuristic 
explanation). The proofs of the upper bounds 
(i.e.~$\lambda_{c,u}\leq C L^{-2}$  and $\lambda_{c,r}\geq C L^{-1}$)
will be based on an explicit construction of the unbounded components,
comparing them to so-called oriented percolation (see the beginning of
Section \ref{sec:upperbounds} for some heuristics).

The rest of this paper is organized as follows. In Section 
\ref{sec:model} we introduce necessary notation and 
define and explain the models and setup. In Section \ref{sec:altstates}
we will state the stronger versions of Theorems \ref{thm:uniform}
and \ref{thm:rigid}.
In Section \ref{sec:lowerbounds} we will prove the 
lower bounds of these two theorems, while the upper bounds will be 
proven in Section \ref{sec:upperbounds}. 
The arguments of Sections \ref{sec:lowerbounds} and 
\ref{sec:upperbounds} (in particular the latter) will rely on 
some fairly long calculations. In order not to interrupt 
the flow of reading more than necessary, these calculations have 
been put in an appendix (Appendix \ref{app:main}).

\section{The Poisson stick model} \label{sec:model}
In this section, we will give formal definitions of the general 
Poisson stick model and the two orientation distributions
considered in this paper.
However, we will start by briefly addressing some basic notation.

We will let $\Vert x \Vert$ 
denote the regular $L^2$-norm of $x\in \BR^d,$ 
and for $A,B\subset \BR^d$ we will write 
$\dist(A,B):=\inf_{x\in A,y\in B} \Vert x-y\Vert $ for the 
distance between the sets $A,B.$  
For any set $A \subset \BR^d$ we let $\Vol(A)$ denote the $d$-dimensional 
Lebesgue measure of $A,$ and we will let   
\begin{equation} \label{eqn:A+def}
A^{+a}:=\{x\in \BR^d: \dist(x,A)\leq a\},
\end{equation}
denote an enlargement of $A.$

Throughout, $B(x,r)\subset \BR^d$ 
will refer to a (closed) ball centered at $x$ and with radius $r>0$. 
In a few places
we will work in $\BR^d$ but consider balls in $\BR^{d-1}.$ In these
places we will write $B_{d-1}(x,r)$ in order to emphasize that it is 
a subset of $\BR^{d-1}.$ Furthermore, $\Vol(B_{d-1}(x,r))$ will then 
refer to the $(d-1)$-dimensional volume of said ball.

We will let  $e_1,\ldots,e_d$ denote the standard unit 
directions in $\BR^d$ while we let 
$\langle x,y\rangle$ denote the usual scalar product of 
$x,y\in \BR^d.$ Furthermore, for any $x=(x_1,x_2,\ldots,x_d)\in \BR^d$ 
we will let 
\begin{equation}\label{eqn:defxr}
x^r=(r,x_2,\ldots,x_d).
\end{equation}
Throughout, we will use $o\in \BR^d$ and $o\in\BZ^d$ to denote the origin.

\subsection{The general Poisson stick model}  \label{sec:stickmodel}

We start by considering the space $\BR^d \times \BS$ where 
$\BS=\{p\in \BR^d: \Vert p\Vert =1\}$ is the unit sphere in $\BR^d.$ 
Then, for any pair $(x,p)\in \BR^d \times \BS$ we associate
the line segment 
\[
\ell_{x,p,L}=\{x+tp : -L/2\leq t\leq L/2\}.
\]
This is a line segment of length $L$ with orientation vector $p$
centered at $x.$ 
Then, we define 
\begin{equation} \label{eqn:defstick}
S_{x,p,L}=\{y\in \BR^d: \dist(y,\ell_{x,p,L})\leq 1\},
\end{equation}
so that $S_{x,p,L}=\ell_{x,p,L}^{+1}.$
We will refer to $S_{x,p,L}$ as a {\em stick}. 
Note that the ``tips'' of the sticks are rounded, and so the sticks 
are not truncated cylinders. This is a matter of convenience 
and will have no qualitative effect on the results of this paper
(although it may affect the constants involved in the bounds of 
Theorems \ref{thm:uniform} and \ref{thm:rigid}). Indeed, one can 
easily sandwich a stick in between two truncated cylinders of lengths $L$
and $L+2$ respectively. Clearly, $S_{x,p,L}=S_{x,-p,L},$ but this will 
not be an issue.

Next we define the intensity measure $\mu_\lambda$ on $\BR^d \times \BS$
which we will use for our Poisson point process. Let 
\begin{equation} \label{eqn:mudef}
\mu_\lambda(\d x,\d p)= \lambda \Leb(\d x) \otimes \Theta(\d p),
\end{equation}
where $\lambda>0$ is a parameter, $\Leb(\d x)$ denotes Lebesgue 
measure on $\BR^d$ and $\Theta(\d p)$ denotes a probability distribution 
on $\BS.$ Next, let 
\[
\Omega=\{\omega \subset \BR^d \times \BS: 
|\omega \cap (A \times \BS)|<\infty 
\textrm{ whenever } \Leb(A)<\infty\},
\]
be the space of configurations.  We will let $\Pi^\lambda$
denote a Poisson point process on $\BR^d \times \BS$ with intensity 
measure $\mu_\lambda$. Since $\Theta(\BS)=1,$ it follows that 
$\Pi^\lambda$ is a random element of $\Omega.$  
We will write $(x,p)\in \Pi^\lambda$ for a point in $\Pi^\lambda.$
Clearly, $\Pi^\lambda$ induces a Poisson point process of sticks in 
$\BR^d$ by identifying a point $(x,p)$ with the stick $S_{x,p,L}.$ 
Consider then 
\[
\CC(\Pi^\lambda)=\bigcup_{(x,p)\in \Pi^\lambda}S_{x,p,L},
\]
which is referred to as the {\em occupied} set. We say that 
{\em percolation} occurs if $\CC(\Pi^\lambda)$ contains a connected 
unbounded component. Following Section 2.1  in \cite{MR_96},
one can use ergodicity to prove that percolation is a $0$-$1$ event.
It is therefore natural to define the critical threshold
$\lambda_{c,\Theta}=\lambda_{c,\Theta}(L)$ for which percolation 
occurs by letting
\begin{equation} \label{eqn:lambdacdef}
\lambda_{c,\Theta}
:=\inf\{\lambda>0: \BP(\CC(\Pi^\lambda) \textrm{ percolates})=1\}.
\end{equation}

\subsection{The uniform and the rigid cases}
Let $\CH(\d p)$ denote the $(d-1)$-dimensional normalized
Hausdorff measure on $\BS$ so that 
\begin{equation} \label{eqn:norm}
\CH(\BS)=\int_{\BS} \CH(\d p)=1.
\end{equation}
Clearly, $\CH(\d p)$ corresponds to uniform distribution of the sticks,
and in this case we let $\lambda_{c,u}$ denote the quantity defined
by \eqref{eqn:lambdacdef}.

The second case we consider in this paper is when all sticks are oriented
in the same direction. Clearly, the choice of direction is not important
and so for definiteness we will consider $\Theta(\d p)=\delta_{e_2}.$
Here, we let $\lambda_{c,r}$ denote 
the quantity corresponding to \eqref{eqn:lambdacdef}.

\section{Alternative statements} \label{sec:altstates}
We will now state the stronger versions of Theorems \ref{thm:uniform}
and \ref{thm:rigid} mentioned in the introduction. 

The statement of Theorem \ref{thm:uniform2} (which is the stronger
version of Theorem \ref{thm:uniform}) contains a lower and
an upper bound. The lower bound holds for any orientation distribution 
while the upper bound holds for any orientation
distribution $\Theta(\d p)$ such that 
$\Theta(\d p)=\phi(p) \CH(\d p)$ where 
for some $\delta>0,$ 
\begin{equation} \label{eqn:phidelta}
\phi(p)\geq \delta \ \CH\textrm{-almost surely.} 
\end{equation}
Since $\Theta(\d p)$ is a probability distribution we clearly have that 
$\int_{\BS} \phi(p) \CH(\d p)=1.$ 

\begin{theorem} \label{thm:uniform2}
For any $d\geq 2$ and {\em any} orientation distribution 
$\Theta(\d p)$ we have that 
\[
\lambda_{c,\Theta} 
\geq \frac{\Gamma((d+1)/2)}{\pi^{(d-1)/2}2^{d}} L^{-2},
\]
for any $L>\pi$.
Furthermore, for any $d\geq 2$ and any orientation distribution 
such that $\Theta(\d p)=\phi(p) \CH(\d p)$ where $\phi(p)$ satisfies 
\eqref{eqn:phidelta} for some $\delta>0,$ we have that 
\[
\lambda_{c,\Theta}
\leq \frac{20 \left(1000 \sqrt{d}\right)^d \sqrt{d}\Gamma(2d-1)}
{9\delta 2^{5(d-2)}\pi^{d/2-2}\Gamma(d/2)^3}L^{-2}
\]
for $L>200\sqrt{d}.$
\end{theorem}
\noindent
{\bf Remark:} The requirement that $L>200\sqrt{d}$ can be improved. 
However, this would 
mean more technical details in the proofs, and since we are interested
in the asymptotics, this does not seem worthwhile.
\medskip

\begin{proof}[Proof of Theorem \ref{thm:uniform} from 
Theorem \ref{thm:uniform2}]
This is a trivial consequence of Theorem \ref{thm:uniform2} by 
considering 
the case when $\phi(p)\equiv 1$.
\end{proof}

Before turning to our second result, we will provide the 
back-of-the-envelope calculation showing why $L^{-2}$ should be the 
correct scaling in Theorem \ref{thm:uniform}.
We will be very informal. 
\begin{itemize}
\item[Step 1:]Start by considering an ``almost-horizontal'' stick.
Since that stick 
is of length $L,$ there should be order $L$ independent 
segments where an ``almost-vertical'' stick can hit it. 

\item[Step 2:] Most almost-vertical sticks that are within range to hit 
the first stick,
have their centers at distance order $L$ from the first stick. Ignoring
all other sticks, such a stick must by spherical symmetry have 
probability of order $L^{-d+1}$ to hit one of the independent segments
on the first stick.

\item[Step 3:] The volume at which the center of the almost-vertical 
stick can be located and still hit the almost-horizontal stick must be of
order $L^d.$ The constant in front of $L^d$ will be small. How small
depends on the exact definition of almost-vertical.

\item[Step 4:] We see that the expected number of almost-vertical sticks
hitting the almost-horizontal is then of order 
$\lambda \cdot L \cdot L^{-d+1} \cdot L^d=\lambda L^{2}.$ Thus, if
$\lambda$ equals a large constant times $L^{-2}$ we should have a good
chance of finding a pair of sticks forming an $L$-shaped figure. 
With this basic building block we should 
be able to construct an 
unbounded component. From this we deduce that the 
correct scaling of $\lambda_{c,u}(L)$ must be $L^{-2}.$ 
 
\end{itemize}
There are of course several issues with the above calculations. 
For instance, we are ignoring everything which is not almost-horizontal 
or almost-vertical. Perhaps this is giving up too 
much? In addition, we ignore almost-vertical sticks whose centers 
are closer to the almost-horizontal than order $L.$ Perhaps these 
are essential since they 
have a better chance of hitting the almost-horizontal stick? 
The answers to both of these questions are (in light of the results of
this paper) no, but clearly some care will be needed in order to provide
a rigorous proof.

We now turn to the stronger version of Theorem \ref{thm:rigid} that includes 
explicit values of the constants involved. Recall that the rigid model 
is when all sticks are oriented along the same direction, and that we chose
$\Theta(\d p)=\delta_{e_2}$ for definiteness.
\begin{theorem} \label{thm:rigid2}
For any dimension $d\geq 2,$ and for every $L>3,$ we have that 
\[
\lambda_{c,r}(L) 
\geq 
\frac{\Gamma((d+1)/2)}{2^d\pi^{d/2}} L^{-1}.
\]
Furthermore, for every $L>10,$ we have that 
\[
\lambda_{c,r}(L) 
\leq 4 \frac{2^d \Gamma((d+1)/2)}{\pi^{d/2-1}}L^{-1}.
\]
\end{theorem}
\noindent
{\bf Remark:} We see by comparing the lower bounds of Theorems 
\ref{thm:uniform2} and \ref{thm:rigid2}, that even though the lower bound 
of Theorem \ref{thm:uniform2} holds
for any orientation distribution, it does not mean that it is always 
a good bound.

\medskip

We will consider the upper and lower bounds in Theorems \ref{thm:uniform2}
and \ref{thm:rigid2} in separate sections. It makes sense to group 
the bounds in the same direction together, as the proof techniques 
are similar.

We end this section with a short discussion on the bounds provided by 
Theorems \ref{thm:uniform2} and \ref{thm:rigid2}. Considering that 
the constants in the upper bounds and the lower bounds of these theorems
differ so greatly, one realizes that at least some of them (and 
maybe all of them) must be far from optimal. Considering the 
somewhat complicated expressions of the involved constants, one may 
also question the value of providing these explicit expressions at all.
The reason for doing so is in part that it makes it easier to follow 
the proofs, since we never have to write things like ``for $c$ small 
enough'' or ``for $L$ large enough'' (except when providing heuristics). 
It also makes it easier for the 
reader to verify just how poor the bounds are. For instance,
one could compare these to numerical bounds that may be provided 
in the future.

\section{The lower bounds} \label{sec:lowerbounds}
The lower bounds of Theorems \ref{thm:uniform2} and \ref{thm:rigid2} will 
be proven along the same lines. We will begin this section by giving an 
informal explanation of the involved argument. 

To that end, consider a single stick placed at the origin (the 
orientation will not be important here). Then, explore the Poisson 
point process in order to 
find any sticks intersecting this first stick. The next step in turn 
then consists of exploring those sticks found in the previous step, and
so on. This stick exploration procedure will be coupled to a Galton-Watson 
process. Elementary measure calculations (which can be found in 
Appendix \ref{app:4.1}) will show that this Galton-Watson process is
subcritical whenever $\lambda<c L^{-2}$ and $c=c(d)>0$ small enough in the uniform 
case, while it is subcritical whenever $\lambda<c L^{-1}$ and $c=c(d)>0$ small 
enough in the rigid case.
For such $\lambda,$ we can conclude that the component of
the original stick at the origin will be finite almost surely, 
and the lower bounds will follow.

We now turn to the proof of the lower bound of Theorem \ref{thm:uniform2}. 
This will require us to know the measure of the set of line segments that 
hit a ball of radius $2.$ However, we will prove the statement for any 
radius $\rho>0$ as this can be done with no extra effort.
The proof of this result (Lemma 
\ref{lemma:msrenumberofballs}) is postponed until Appendix 
\ref{app:4.1}.
\begin{lemma}[\bf Measure of line segments hitting a ball] \label{lemma:msrenumberofballs}
For any $d\geq 2,$ $L,\rho>0$ and any distribution $\Theta(\d p)$ 
we have that
\begin{eqnarray*}
\lefteqn{\mu_\lambda((x,p)\in \BR^d \times \BS: 
\ell_{x,p,L}\cap B(o,\rho)\neq \emptyset)}\\
& & =\Vol(S_{o,p,L}(\rho))
=L\frac{\pi^{(d-1)/2}}{\Gamma((d+1)/2)}\rho^{d-1}
+\frac{\pi^{d/2}}{\Gamma(d/2+1)}\rho^d.
\end{eqnarray*}
\end{lemma}
\medskip

\noindent
Equipped with this lemma we can now prove the lower bound of 
Theorem \ref{thm:uniform2}. Throughout most of this section, the 
location of the 
center point $x$ and the direction $p$ of 
a line segment $\ell_{x,p,L}$ will not explicitly be used. It is therefore  
convenient to abuse notation and simply write $\ell \in \Pi^\lambda$ 
for a line segment corresponding to some point in the Poisson
point process,
and $S_\ell$ for the corresponding stick. For $A\subset \BR^d,$ 
we will also write 
$\mu_\lambda(\ell: S_l \cap A \neq \emptyset)$ instead of 
$\mu_\lambda((x,p)\in \BR^d \times \BS: S_{x,p,L}\cap A \neq 
\emptyset).$ \\

\begin{proof}[Proof of lower bound of Theorem \ref{thm:uniform2}]
Although the idea behind the proof is fairly straightforward, some care
is needed when constructing the actual coupling. In order to explain
the idea, fix some line segment $\ell_0,$ let $\Pi^\lambda$ be as in Section 
\ref{sec:stickmodel},
and let $(\Pi^\lambda_{k,n})_{k,n\geq 1}$ be an i.i.d.~collection 
of random variables which are independent copies of $\Pi^\lambda.$
Then, we let $\CC_0(\Pi^\lambda)$ denote the connected component
(if it exists) of $\CC(\Pi^\lambda)$ such that 
$S_{\ell_0}\cap \CC_0(\Pi^\lambda)\neq \emptyset.$ 
The aim is to prove that $\CC_0(\Pi^\lambda)$ is a bounded set almost 
surely whenever $\lambda<c L^{-2}$ where $c=c(d)$ is small enough. 
We will prove this by stochastically comparing $\CC_0(\Pi^\lambda)$ 
to something larger, but which is easier for us to analyze. 
By our construction, this larger object
will naturally correspond to a subcritical branching process. 
We can then conclude that this larger object must be bounded, and from 
this it will immediately follow that $\CC(\Pi^\lambda)$ cannot contain 
any unbounded connected components.

We start by letting
\[
\Psi^\lambda_1
=\{\ell \in \Pi^\lambda: S_\ell\cap S_{\ell_0}\neq \emptyset\},
\]
so that $\Psi^\lambda_1$ consists of those line segments in 
$\Pi^\lambda$ whose corresponding sticks intersect $S_{\ell_0}.$ 
Let $\ell_{1,1},\ldots,\ell_{|\Psi_\lambda^1|,1}$ be an enumeration 
of the line segments in $\Psi_\lambda^1,$ and we think of these as the 
line segments of generation 1. 

Next, we let 
\[
\Psi^\lambda_{1,2}=\{\ell \in \Pi^\lambda \setminus \Psi^\lambda_1: 
S_\ell\cap S_{\ell_{1,1}}\neq \emptyset\}
\cup
\{\ell \in \Pi^\lambda_{1,1}: S_\ell\cap S_{\ell_{1,1}}\neq \emptyset,
S_\ell \cap S_{\ell_0} \neq \emptyset\}.
\]
The first set of line segments in this union are the segments 
$\ell\in \Pi^\lambda$ 
such that $S_\ell\cap S_{\ell_{1,1}}\neq \emptyset,$ but that we did 
not already encounter when defining $\Psi^\lambda_1$ (i.e.~they did not 
intersect $S_{\ell_0}$). The second set of line segments are 
then ``extra'' segments from $\Pi^\lambda_{1,1}$ which are required to 
hit both $S_{\ell_{1,1}}$ and $S_{\ell_0}.$ The reason for adding these 
extra line segments is
that now, $\Psi^\lambda_1$ and $\Psi^\lambda_{1,2}$ are both
generated by considering a Poisson point process of line segments intersecting 
$S_{\ell_0}^{+1}$ and $S_{\ell_{1,1}}^{+1}$ (recall \eqref{eqn:A+def})
respectively.
Therefore, 
$|\Psi^\lambda_1|$ and $|\Psi^\lambda_{1,2}|$ are equal in distribution, 
and furthermore,
any $\ell\in \Pi^\lambda$ such that $S_\ell$ intersects $S_{\ell_0}$ or 
$S_{\ell_{1,1}}$ must by construction also belong to 
$\Psi^\lambda_1 \cup \Psi^\lambda_{1,2}.$
For $k=2,\ldots,|\Psi^\lambda_1|$ we then let 
\begin{eqnarray*}
\lefteqn{\Psi^\lambda_{k,2}
=\{\ell\in \Pi^\lambda\setminus 
\left(\Psi^\lambda_1 \cup_{j=1}^{k-1}\Psi^\lambda_{j,2}\right):
S_\ell\cap S_{\ell_{k,1}}\neq \emptyset\} }\\
&& \hspace{8mm} \cup \{\ell \in \Pi^\lambda_{k,1}: 
S_\ell\cap S_{\ell_{k,1}}\neq \emptyset,  
S_\ell \cap \left( S_{\ell_0} \cup_{j=1}^{k-1}S_{\ell_{j,1}}\right)
\neq \emptyset \}.
\end{eqnarray*}
As above, the first set is the set of line segments in $\Pi^\lambda$ 
not yet encountered. We think of this as using $\Pi^\lambda$ wherever 
$\Pi^\lambda$ 
has not already been explored/used. The second part is then 
an ``extra'' set of line segments, i.e.~we use $\Pi^\lambda_{k,1}$ 
wherever $\Pi^\lambda$ has already been used. By adding these segments, 
we compensate for the space already explored,  and so $|\Psi^\lambda_{k,2}|$
where $k=1,2,\ldots,|\Psi_1^\lambda|$ becomes an i.i.d.~sequence
of random variables, all with the same distribution as $|\Psi^\lambda_{1}|.$
Next, we let 
\[
\Psi^\lambda_2=\bigcup_{k=1}^{|\Psi^\lambda_1|}\Psi^\lambda_{k,2},
\]
and note that if $\ell\in \Pi^\lambda$ is such that we can jump from 
$S_{\ell_0}$ to $S_\ell$ by using at most one other stick in 
$\Pi^\lambda,$ 
our construction gives us that 
$\ell \in \Psi^\lambda_1 \cup \Psi^\lambda_2.$
We then enumerate the line segments 
$\ell_{1,2},\ldots,\ell_{|\Psi^\lambda_2|,2}$
in $\Psi_\lambda^2,$ and think of these as the segments of generation 2. 

The general step is performed in the same way. We define the collection
$(\Psi^\lambda_{k,n+1})_{1\leq k \leq |\Psi^\lambda_n|}$ and let 
\[
\Psi^\lambda_{n+1}
=\bigcup_{k=1}^{|\Psi^\lambda_{n}|}\Psi^\lambda_{k,n+1} 
\textrm{ and finally }
\Psi^\lambda=\bigcup_{n=1}^\infty \Psi^\lambda_n.
\]
If we let 
$\CC_0(\Psi^\lambda)=S_{\ell_0}\bigcup_{\ell \in \Psi^\lambda} S_\ell$ 
(which by definition is a connected component) 
we see from our construction that
\begin{equation} \label{eqn:Cinclusion}
\CC_0(\Pi^\lambda) \subset \CC_0(\Psi^\lambda).
\end{equation}

It is not hard to see that the sequence $(|\Psi^\lambda_n|)_{n=1}^\infty$
corresponds to the sizes of the generations of a Galton-Watson 
family tree. In order to show that this is subcritical for small
enough values of $\lambda,$ we must now estimate 
$\BE[|\Psi^\lambda_1|].$ To that end, note that
since the length of $\ell$ is $L,$ we can 
cover any stick $S_\ell$ by using at most $L$ balls of radius 2 
whenever $L\geq 1$ (which holds by assumption).
If $S$ is such a stick and $z_1,\ldots,z_L$ denote the centers of the
balls in such a covering, we see that for some $C<\infty,$
\begin{eqnarray} \label{eqn:muupper}
\lefteqn{\mu_\lambda(\ell: S_\ell \cap S \neq \emptyset)}\\
& & \leq \bigcup_{k=1}^L\mu_\lambda(\ell: \ell \cap B(z_k,2)\neq \emptyset)
= L\mu_\lambda(\ell: \ell \cap B(z_1,2)\neq \emptyset) 
\nonumber \\
& & =L \lambda 
\left(L\frac{\pi^{(d-1)/2}}{\Gamma((d+1)/2)}2^{d-1}
+\frac{\pi^{d/2}}{\Gamma(d/2+1)}2^d\right)
\leq \lambda L^2 \frac{\pi^{(d-1)/2}}{\Gamma((d+1)/2)}2^{d}
\nonumber
\end{eqnarray}
by using Lemma \ref{lemma:msrenumberofballs}
with $\rho=2$ and where the last inequality holds for any 
\begin{equation} \label{eqn:Llower}
L\geq 2 \sqrt{\pi}\frac{\Gamma((d+1)/2)}{\Gamma(d/2+1)}.
\end{equation}
By using that the gamma function is a logarithmically convex function,
it is easy to show that the right-hand side of 
\eqref{eqn:Llower} is decreasing in $d\geq 2.$ Furthermore, the
expression equals $\pi$ for $d=2$ and since we assume that $L>\pi$ 
we conclude that \eqref{eqn:Llower} holds, and in turn that 
\eqref{eqn:muupper} holds. Using this, we then see that
\begin{equation} \label{eqn:expequal}
\BE[|\Psi^\lambda_1|]
=\mu_\lambda(\ell: S_\ell \cap S \neq \emptyset)
\leq \lambda L^2 \frac{\pi^{(d-1)/2}}{\Gamma((d+1)/2)}2^{d}
\end{equation}
which is strictly smaller than one whenever 
\[
\lambda<\frac{\Gamma((d+1)/2)}{\pi^{(d-1)/2}2^{d}}L^{-2}.
\]
For such values of $\lambda,$ the corresponding branching process is 
subcritical and therefore it dies out almost surely. Then, 
we conclude that $|\Psi^\lambda|<\infty$ almost surely, and 
so by \eqref{eqn:Cinclusion} we must also have that $\CC_0(\Pi^\lambda)$
is bounded almost surely. 

Since the choice of $S_{\ell_0}$ was arbitrary, it follows by standard 
Poisson point process theory that $\CC(\Pi^\lambda)$ cannot contain any 
unbounded connected components, and so we conclude that 
\[
\lambda_{c,\Theta} \geq \frac{\Gamma((d+1)/2)}{\pi^{(d-1)/2}2^{d}} L^{-2}.
\]
\end{proof}

Next we turn to the rigid case.

\begin{proof}[Proof of lower bound of Theorem \ref{thm:rigid2}]
The proof is very similar to the proof of the lower bound 
of Theorem \ref{thm:uniform2}; the
only difference lies in the estimate of 
$\mu_\lambda(\ell: S_\ell\cap S\neq \emptyset)$ in
 \eqref{eqn:expequal}.
In order to estimate 
this, consider a stick $S_{o,e_1,L}.$ If
$S_{x,e_1,L} \cap S_{o,e_1,L}\neq \emptyset,$ it is necessary 
(but not sufficient) that 
$x_1\in [-L-2,L+2]$  and that $\Vert x^0 \Vert < 2$ (recall 
\eqref{eqn:defxr}).
Note that the reason that it is not $x_1\in [-L,L]$ is due to 
the ``tips'' of the sticks. We therefore see that (recall the 
notation $B_{d-1}(x,r)$ from Section \ref{sec:model})
\begin{eqnarray*}
\lefteqn{
\BE[|\Psi_1^\lambda|]
=\mu_\lambda ((x,e_1):S_{x,e_1,L} \cap S_{o,e_1,L}\neq \emptyset)}\\
& & \leq \lambda (2L+4)\Vol(B_{d-1}(o,2))
= \lambda (2L+4)\frac{\pi^{(d-1)/2}}{\Gamma((d-1)/2+1)}2^{d-1} \\
& & \leq \lambda 2^d\frac{\pi^{d/2}}{\Gamma((d+1)/2)}L,
\end{eqnarray*}
where the last inequality holds for $L>3.$
Replacing the estimate in \eqref{eqn:expequal} with this expression 
gives us the result.
\end{proof}

\section{The upper bounds} \label{sec:upperbounds}
In this subsection we will provide proofs of the upper bounds of 
Theorems \ref{thm:uniform2} and \ref{thm:rigid2}.
Here, we will couple the 
stick process with a so-called oriented percolation model. We
will attempt to construct an unbounded component, and the coupling
will then show that this construction 
has positive probability of succeeding when $\lambda>C L^{-2}$ 
(with a suitable choice of $C=C(d)<\infty$) for the 
uniform case. A similar construction and coupling will work 
analogously whenever $\lambda>C L^{-1}$ (again with a suitable 
choice of $C=C(d)<\infty$) in the rigid case.
This construction will require measure calculations 
which are mainly postponed until Appendix \ref{app:5.1}.

Since both proofs are done through a comparison with the oriented 
percolation model (although using slightly different variants in the 
two cases), we will start by introducing said model. To that end,
consider the following lattice in the upper half-plane,
\begin{equation} \label{eqn:Hdef}
\BH:=\{u\in \BZ^2: u_1+u_2 \textrm{ is even, and } u_2 \geq 0\}.
\end{equation}
We will consider a collection $(A_n)_{n \geq 0}$ of random subsets of 
$\BH$ where $A_0=\{o\}$ and where $A_n$ is such that 
\[
A_n \subset \{u \in \BH: u_2=n\}.
\]
Given $A_n,$ the events $\{u \in A_{n+1}\}$
are conditionally independent with  
\begin{equation} \label{eqn:OPdef}
\BP(u\in A_{n+1}|A_{n})=\left\{
\begin{array}{cl}
 \beta &\textrm{ if } |A_n\cap\{u+(-1,-1),u+(1,-1)\}|=2 \\
\alpha & \textrm{ if } |A_n\cap\{u+(-1,-1),u+(1,-1)\}| =1\\
0 &\textrm{ if } |A_n\cap\{u+(-1,-1),u+(1,-1)\}| =0.
\end{array}
\right.
\end{equation}
Here, $0<\alpha,\beta<1,$ and we will consider two variants. In the first, 
$\beta=1-(1-\alpha)^2$ and in the second, $\beta=\alpha.$

We will give an informal description of this model and the two variants.
Consider the first variant where $\beta=1-(1-\alpha)^2.$ With
$A_0=\{o\},$ we obtain from \eqref{eqn:OPdef} that
for $u=(-1,1)$ and $u=(1,1)$ we have that 
\[
\BP(u\in A_1)=\alpha
\] 
independently for $u=(-1,1)$ and $u=(1,1).$ We can think of this as
drawing an arrow with probability $\alpha$ from 
$o$ to $(-1,1)$ and independently, with the same probability, drawing
one from $o$ to $(1,1).$  In general, conditional on 
$A_n,$ then if $u\in A_n,$ we draw an arrow with probability 
$\alpha$ from $u$ to $u+(1,1),$ and independently (again with probability 
$\alpha$), from $u$ to $u+(-1,1)$. This is done independently for every 
$u\in A_n.$ Then, every $v\in \BH$ which is pointed at by an arrow
emanating from some $u\in A_n$ belongs to $A_{n+1}.$ Clearly, this is 
a bond percolation model. The second variant (i.e.~when $\beta=\alpha$)
is in contrast a site percolation model since here 
$\BP(u\in A_{n+1}|A_n)=\alpha$ whether there are one or two sites in 
$A_n$ ``preceeding'' $u.$

The bond version of this model was introduced in \cite{D_84} where it is 
proven that if $\alpha<1$ is large enough, 
\begin{equation} \label{eqn:OPsurvives}
\BP(A_n \neq \emptyset \ \ \forall n\geq 0)>0.
\end{equation}
In order for us to obtain the explicit bounds required in Theorems 
\ref{thm:uniform2} and \ref{thm:rigid2}, we shall need some explicit 
bounds on $\alpha$ such that \eqref{eqn:OPsurvives} holds, and 
we require this for both variants. 
It was proven in \cite{L_95} that for the bond version, any $\alpha>2/3$
is such that \eqref{eqn:OPsurvives} holds, while for the site version,
\eqref{eqn:OPsurvives} holds whenever $\alpha>3/4.$

We will use the bond model to prove the upper 
bound of Theorem \ref{thm:uniform2}, although we do this on the 
lattice $2\BH$ 
(this is a matter of notational convenience and clearly does not 
matter). When proving the upper bound of 
Theorem \ref{thm:rigid2}, we instead use the site model on $\BH.$

\subsection{The upper bound of Theorem \ref{thm:uniform2}}
\label{sec:mainupperbound}
We start this subsection by providing the intuition along with some
necessary notation. The proof will rely on a construction of an
unbounded component which essentially will be performed on a two-dimensional 
lattice. 
To that end, for any $u\in \BZ^2$ we will write
\begin{equation} \label{eqn:defDv}
D^u=(u_1,u_2,0,\ldots,0)\frac{L}{4}
+\left[-\frac{L}{16\sqrt{d}},\frac{L}{16\sqrt{d}}\right]^d,
\end{equation}
where $(u_1,u_2,0,\ldots,0)\in \BZ^d.$ Note that 
$D^u\cap D^v=\emptyset$ whenever $u \neq v.$
We then start by considering the three boxes $D^{(-2,0)}, D^{(-1,0)}$ 
and $D^{o}.$ For any $\lambda>0$ there is a positive probability 
that there exists $(x,p)\in \Pi^\lambda$ with the following three 
properties. Firstly, the center of $\ell_{x,p,L}$ (i.e.~$x$) 
belongs to $D^{(-1,0)}.$ Secondly, $\ell_{x,p,L}$ intersects 
$D^{(-2,0)},$ and lastly $\ell_{x,p,L}$ intersects 
the ``right-hand'' boundary of $D^{o}.$ 
The corresponding stick $S_{x,p,L}$ will then present a 
{\em target}
for a second stick (see Figure \ref{fig:Target1}
for an illustration). 
This second stick $S_{y,q,L}$ will be required to have its center 
in the box $D^{(0,1)},$ to hit the target presented by $S_{x,p,L}$ and
to intersect the ``top'' boundary of $D^{(0,2)}$ 
(see Figure \ref{fig:Target2}).
This in turn then becomes a target for 
{\em two} additional sticks centered in $D^{(-1,2)}$ and $D^{(1,2)},$ 
presenting targets in $D^{(-2,2)}$ and $D^{(2,2)}$ respectively. 
This procedure can be continued and coupled 
with the oriented percolation model described above.

\begin{figure}
\centering
\begin{subfigure}{.5\textwidth}
  \centering
  \includegraphics[clip=true, trim=0 0 280 5, width=1\linewidth, page=1]{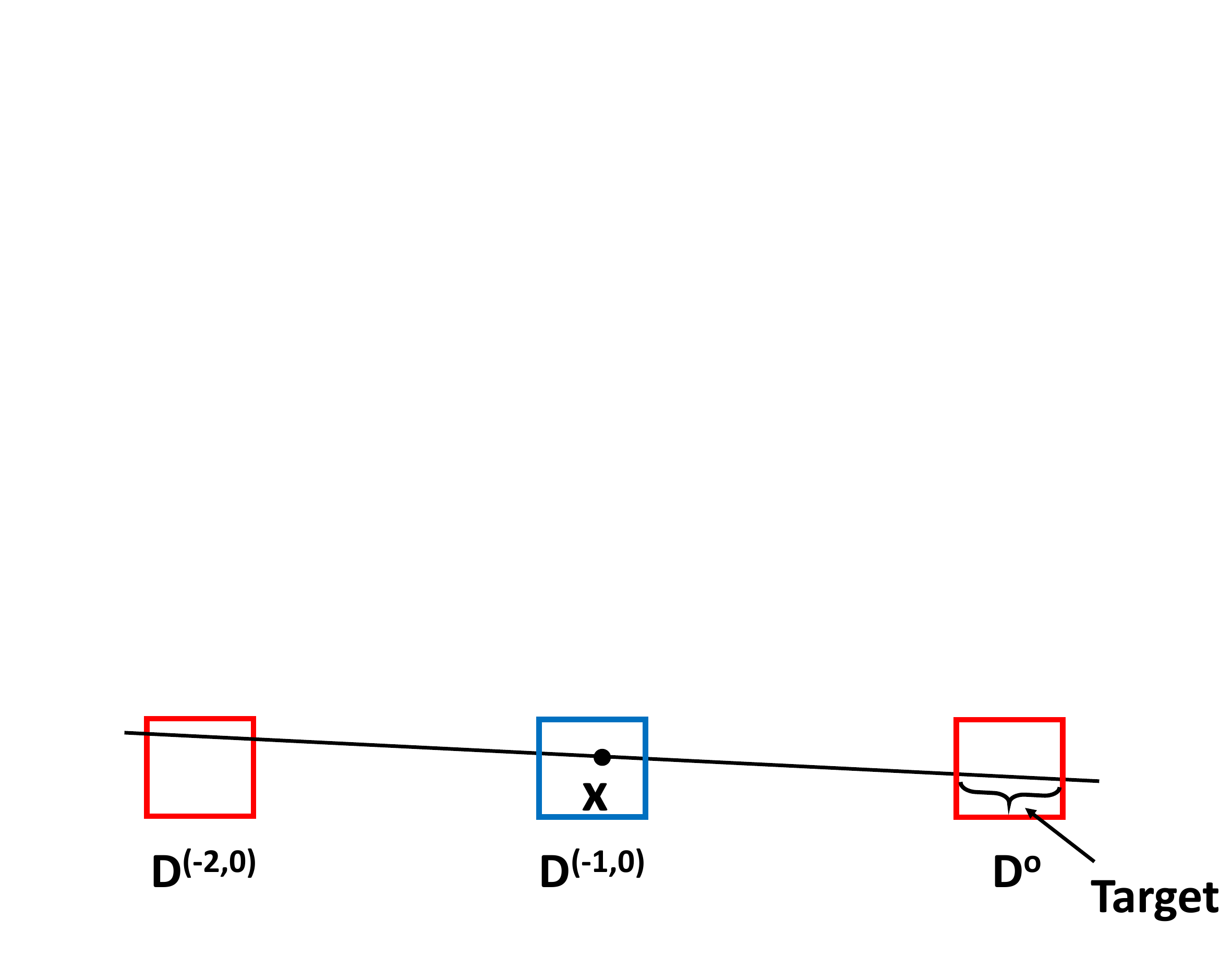}
  \caption{}
  \label{fig:Target1}
\end{subfigure}%
\begin{subfigure}{.5\textwidth}
  \centering
  \includegraphics[clip=true, trim=0 0 280 5, width=\linewidth,page=2]{Targets.pdf}
  \caption{}
  \label{fig:Target2}
\end{subfigure}
\caption{In (a) the stick $S_{x,p,L}$ is centered in $D^{(-1,0)}$,
intersects $D^{(-2,0)}$ and ``exits'' the right-hand side of
$D^{o}$. It presents a target for the next step.
In (b) the stick $S_{y,p,L}$ hits the previous target and presents
a new one. Here, blue boxes are those where centers of our sticks 
will be found, while red boxes are where intersections occur.} 
\label{fig:Target}
\end{figure}

Our next step is to introduce 
the following notation which will make the notion of ``right-hand'',
``left-hand'' and ``top'' part of the boundary of the boxes $D^u$ 
precise. For any $u\in \BZ^2,$ let 
\begin{align*}
L^{-16}(D^u)
& =\Bigg\{x\in D^{u}: x_1=u_1\frac{L}{4}-\frac{L}{16\sqrt{d}},
\left|x_2-u_2\frac{L}{4}\right|\leq \frac{L}{16\sqrt{d}}-16 \\
& \hspace{60 mm} \textrm{ and }\left|x_k\right|\leq \frac{L}{16\sqrt{d}}-16  
\textrm{ for } k=3,\ldots,d \Bigg\}, \\
R^{-16}(D^u)
& =\Bigg\{x\in D^{u}: x_1=u_1\frac{L}{4}+\frac{L}{16\sqrt{d}},
\left|x_2-u_2\frac{L}{4}\right|\leq \frac{L}{16\sqrt{d}}-16 \\
& \hspace{60 mm} \textrm{ and }\left|x_k\right|\leq \frac{L}{16\sqrt{d}}-16  
\textrm{ for } k=3,\ldots,d \Bigg\}, \\
\textrm{  and } \\
T^{-16}(D^u) 
&=\Bigg\{x\in D^{u}: x_2=u_2\frac{L}{4}+\frac{L}{16\sqrt{d}},
\left|x_1-u_1\frac{L}{4}\right|\leq \frac{L}{16\sqrt{d}}-16 \\
& \hspace{60 mm} \textrm{ and }\left|x_k\right|\leq \frac{L}{16\sqrt{d}}-16  
\textrm{ for } k=3,\ldots,d \Bigg\}. \\
\end{align*}
We see that $L^{-16}(D^u),R^{-16}(D^u)$ and $T^{-16}(D^u)$ are subsets 
of the left-hand, right-hand and top part of the 
boundary of $D^u$ (see also Figure \ref{fig:Dubox}).
We will require the line segments of our construction to hit these sets 
(rather than anywhere on the corresponding faces of $D^u$). The reason 
for this is that unless the line segments hit ``well inside'' the 
faces, then they will not behave in a way that makes a continuation 
of the construction possible.
We note that we will not need a notation for the bottom part. 

\begin{figure} 
\includegraphics[clip=true, trim=0 150 0 10, width=16.5cm, page=2]{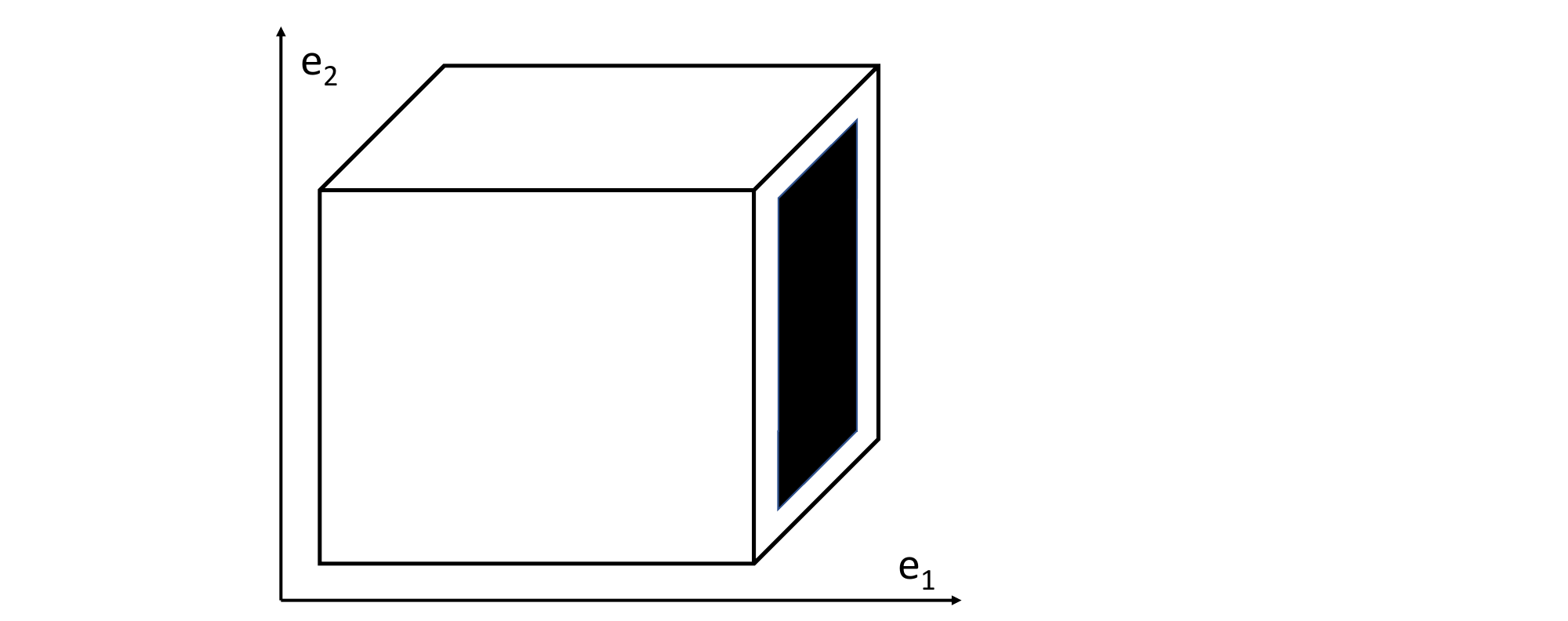} 
\caption{We see a box $D^u$ centered at $(u_1,u_2,0)L/4$ and with side 
length $\frac{L}{8 \sqrt{3}}.$ We can also see $R^{-16}(D^u)$ as the solid
black area on the right-hand side of the box.} \label{fig:Dubox}
\end{figure}

For two sets $A,B\subset \BR^d$ we will write 
\begin{equation} \label{eqn:conndef}
A \xleftrightarrow{\ell_{x,p,L}} B
\end{equation}
for the event $\{\ell_{x,p,L}\cap A \neq \emptyset\}
\cap\{\ell_{x,p,L}\cap B \neq \emptyset\}.$  Thus, 
$B(\gamma,\rho) \xleftrightarrow{\ell_{x,p,L}} B(\zeta,\rho)$
is the event that the line 
segment $\ell_{x,p,L}$ connects the two balls $B(\gamma,\rho)$ and 
$B(\zeta,\rho).$ 
Let $\gamma\in D^{(-2,0)},$ be such that $B(\gamma,2)\subset D^{(-2,0)}$
and let $\zeta \in D^o.$ Note that 
\begin{eqnarray*}
\lefteqn{\{(x,p): x\in D^{(-1,0)}, 
B(\gamma,2)\xleftrightarrow{\ell_{x,p,L}} B(\zeta,2)\}}\\
& & =\{(x,p): x\in D^{(-1,0)}, B(\gamma,1)\cap S_{x,p,L} \neq \emptyset,
B(\zeta,1)\cap S_{x,p,L} \neq \emptyset \},
\end{eqnarray*}
since the ball $B(\gamma,1)$ and the stick 
$S_{x,p,L}$ touches if and only if $\ell_{x,p,L}$ comes within 
distance 2 from $\gamma.$ (This is where our choice of definition of 
a stick with rounded tips, i.e.~\eqref{eqn:defstick}, becomes convenient.)

Our next lemma is proved through a number of intermediate steps,
and we postpone the proofs of these and the lemma itself until Appendix 
\ref{app:5.1}.

\begin{lemma}[\bf Measure of line segments connecting two balls]
\label{lemma:twoballmeasure2}
Let $\Theta(\d p)=\phi(p) \CH(\d p)$ where $\phi(p)$ satisfies
\eqref{eqn:phidelta} for some $\delta>0.$
For every $d\geq 2, L>32$ and any 
$\gamma \in D^{(-2,0)},$ $\zeta \in R^{-16}(D^o)$ we have that 
\[
\mu_\lambda\left((x,p): x\in D^{(-1,0)}, 
B(\gamma,2)\xleftrightarrow{\ell_{x,p,L}} B(\zeta,2)\right) 
\geq \lambda \delta c_d L^{-d+2},
\]
where we may take 
\[
c_d=\frac{2^{5(d-2)} \pi^{d/2-2}}{\sqrt{d}} 
\frac{\Gamma(d/2)^3}{\Gamma(2d-1)}.
\]
\end{lemma}

Lemma \ref{lemma:twoballmeasure2} will be used to
estimate the probability of hitting a target and 
simultaneously providing a new one (as discussed at the start of this 
subsection). However, before we 
are ready to obtain such an estimate, we need to make precise what 
these notions mean. Recall therefore \eqref{eqn:A+def} and let
$D^{u,+2}:=(D^u)^{+2}$ be an enlargement of $D^u.$
Note that $D^{u,+2}$ is not the same as $D^{u+(2,2)},$ since the latter
refers to the box defined by \eqref{eqn:defDv} with 
$u=(u_1,u_2)$ so that $u+(2,2)=(u_1+2,u_2+2).$
Note also that while 
$R^{-16}(D^o)$ denotes a ``shrunk'' version of the right-hand boundary
of $D^o$ (and therefore of a $(d-1)$-dimensional object), $A^{+a}$ 
as defined in \eqref{eqn:A+def}
denotes an enlargement in all dimensions.

Fix some $\ell_{x,p,L}$ such that 
\[
 D^{(-2,0),+2}\xleftrightarrow{\ell_{x,p,L}} R^{-16}(D^{o}).
\]
Given this $\ell_{x,p,L},$ let 
\[
\CT(\ell_{x,p,L},D^o)=\{y\in \BR^d: 
\dist(y,\ell_{x,p,L}\cap D^{o})\leq 1\},
\]
which we think of as the (horizontal) {\em target} presented 
by $S_{x,p,L}$ for the next step. Note that 
$\CT(\ell_{x,p,L},D^o)\subset S_{x,p,L}$ and while 
$\CT(\ell_{x,p,L},D^o)\not \subset D^o$ it is the case that 
 $\CT(\ell_{x,p,L},D^o) \subset D^{o,+1}.$

The key lemma used for the comparison of the stick process and 
the oriented percolation model considers the set of line segments 
$\ell_{y,q,L}$ 
such that $S_{y,q,L}\cap \CT(\ell_{x,p,L},D^o) \neq \emptyset$ and 
such that $\ell_{y,q,L} \cap B(\zeta_y,2)\neq \emptyset$ for some 
$\zeta_y\in  T^{-16}(D^{(0,2)}).$
That is, the corresponding stick $S_{y,q,L}$ connects with the target, 
and also presents a new (vertical) target in the box $D^{(0,2)}$ that can
be used for further connections (see Figure \ref{fig:Target}).
The proof of this lemma considers two collections of balls, and the
proof will use the elementary case in Lemma \ref{lemma:twoballmeasure2}
(which concerns only two balls). 
The first collection will be embedded in the 
target, while the second will be a collection with centers 
embedded in $T^{-16}(D^{(0,2)}).$ Then,
we will sum over these collections to obtain our estimate.
For further steps in the construction, we will need to define 
collections with centers in $L^{-16}(D^{(0,2)}),$ $R^{-16}(D^{(0,2)})$
and $T^{-16}(D^{(0,2)})$ respectively as follows. Let
\begin{align*}
\BL(D^u)& 
= \left\{x\in L^{-16}(D^{u}): (x_2,\ldots,x_d)\in 12 \BZ^{d-1} \right\}, \\
\BR(D^u) &
=\left\{x\in R^{-16}(D^{u}): (x_2,\ldots,x_d)\in 12 \BZ^{d-1} \right\}
\textrm{ and } \\
\BT(D^u) &
=\left\{x\in T^{-16}(D^{u}):(x_1,x_3,\ldots,x_d)\in 12 \BZ^{d-1} \right\}.
\end{align*}
We see that these are discretizations of their respective boundary 
pieces with points separated by a distance of at least 12.
The reason for this separation distance is that certain hitting
events will become independent. In order to show this, we will need
a ``disjointness'' result, namely Lemma \ref{lemma:t1tau1} whose proof
is provided in Appendix \ref{app:5.2}. We also note for future reference
that the number of points $|\BT(D^u)|$ in $\BT(D^u)$ can be bounded 
below by 
\begin{equation} \label{eqn:TDbound}
|\BT(D^u)| \geq \left(\frac{\frac{L}{8\sqrt{d}}-32}{12}-1\right)^{d-1}
\geq \left(\frac{L}{96\sqrt{d}}-4\right)^{d-1},
\end{equation}
since the side length of $T^{-16}(D^u)$ is $\frac{L}{8\sqrt{d}}-32,$
the spacing of the lattice is 12, and then we subtract 1 
for boundary issues.

Similar to the definition of $\ell_{x,p,L}$, let
\[
\ell_{x,p,\infty}=\{x+tp:-\infty<t<\infty\}
\]
and
\[
\ell_{y,q,\infty}=\{y+\tau q:-\infty<\tau<\infty\},
\]
where $p,q$ as before are vectors such that 
$\Vert p\Vert =\Vert q\Vert =1,$ be two 
parametrized (infinite) lines. 
Informally, Lemma \ref{lemma:t1tau1} shows that if the angle 
between two orientation vectors 
$p$ and $q$ is not too small (or equivalently if 
$|\langle p,q\rangle|$ is not too large), then the lines 
$\ell_{x,p,\infty}$ and $\ell_{x,q,\infty}$ will not be 
close for very long. 

\begin{lemma}
\label{lemma:t1tau1}
Assume that $p,q\in \BS$ are such that 
$|\langle p,q\rangle|\leq \frac{1}{\sqrt{2}}$
and that $t_1,\tau_1$ are such that 
\begin{equation} \label{eqn:dsmallerthanrho}
\Vert \ell_{x,p,\infty}(t_1)-\ell_{y,q,\infty}(\tau_1) \Vert
\leq 2.
\end{equation}
Then we have that for every $(t,\tau)$ such that 
$\max(|t-t_1|,|\tau-\tau_1|)\geq 12,$
\[
\Vert \ell_{x,p,\infty}(t)-\ell_{y,q,\infty}(\tau)\Vert\geq 6.
\]
\end{lemma}
\noindent
{\bf Remark:}
If $|\langle p,q \rangle|\leq 1/\sqrt{2},$ and if 
$B(\ell_{y,q,\infty}(\tau_1),2)$ is such that 
$\ell_{x,p,\infty}\cap B(\ell_{y,q,\infty}(\tau_1),2)\neq \emptyset,$
then for some $t_1,$
\[
\Vert \ell_{x,p,\infty}(t_1)-\ell_{y,q,\infty}(\tau_1)\Vert
\leq 2.
\]
It follows from Lemma \ref{lemma:t1tau1} that if
$|\tau_1-\tau_2|\geq 12,$ then for every value of $t,$
\[
\Vert \ell_{x,p,\infty}(t)-\ell_{y,q,\infty}(\tau_2)\Vert\geq 6.
\]
We therefore see that 
\[
\ell_{x,p,\infty}\cap B(\ell_{y,q,\infty}(\tau_2),2)=\emptyset.
\]
Informally, this means that if the line $\ell_{x,p,\infty}$ hits 
a ball of radius 2 centered on the line $\ell_{y,q,\infty},$ it can 
not hit any
other such ball as long as the distance between the centers exceed
$12.$ This is how Lemma \ref{lemma:t1tau1} will be used.

\medskip

We can now state and prove our key lemma.
\begin{lemma} \label{lemma:lineface}
Let $\gamma_x\in D^{(-2,0)}, \zeta_x \in \BR(D^o)$ and $\ell_{x,p,L}$
be such that 
\[
B(\gamma_x,2)\xleftrightarrow{\ell_{x,p,L}} B(\zeta_x,2).
\]
For every $d\geq 2,$ $\gamma_x,\zeta_x$ and $\ell_{x,p,L}$ as above, 
we have that 
\begin{eqnarray*}
\lefteqn{\BP\left(\exists (y,q)\in \Pi^\lambda,\zeta_y\in \BT(D^{(0,2)}): 
y\in D^{(0,1)}, 
\CT(\ell_{x,p,L},D^o)^{+1} \xleftrightarrow{\ell_{y,q,L}}
B(\zeta_y,2)\right)}\\
& & \hspace{110mm}
\geq 1-\frac{1}{c_d' \lambda \delta L^2},
\end{eqnarray*}
for $L>200\sqrt{d}$ and where 
\[
c_d'=\frac{c_d}{\left(1000\sqrt{d}\right)^d}
\]
with $c_d$ is as in Lemma 
\ref{lemma:twoballmeasure2}.
\end{lemma}
\begin{proof}
As mentioned, the proof will rely on Lemma \ref{lemma:twoballmeasure2}
and two collections of balls. Our first collection is simply all 
$B(z,2)$ where $z\in \BT(D^{(0,2)}).$
In order to find our second collection, we start by fixing some 
$\gamma_x,\zeta_x$ and $\ell_{x,p,L}$ as in the assumption. 
Then, let $t_1<t_2<\ldots <t_N$ be such that $\ell_{x,p,L}(t_i)\in D^{o}$ 
for every $i=1,\ldots,N,$ and such that  $|t_i-t_{i+1}|=12$ for every 
$i=1,\ldots, N-1.$ Next, consider any
$z=(z_1,\ldots,z_d)\in\ell_{x,p,L}$ such that
\begin{equation} \label{eqn:z1in}
z_1\in \left[-\frac{L}{16 \sqrt{d}},\frac{L}{16 \sqrt{d}}\right].
\end{equation}
It is a straightforward, although somewhat tedious, exercise in 
trigonometry to show that since 
$\ell_{x,p,L}\cap D^{(-2,0),+2}\neq \emptyset$ and
$\ell_{x,p,L}\cap B(\zeta_x,2)\neq \emptyset,$ we must 
for such a $z$ have that  
\begin{equation} \label{eqn:zkineq}
|z_k| \leq  \frac{L}{16\sqrt{d}} \textrm{ for } k=1,2,\ldots,d.
\end{equation}
We will leave the details of this fact to the reader, although we 
remark that we outline the argument for a similar statement in the 
proof of Lemma \ref{lemma:twoballmeasure2} in Appendix \ref{app:5.1}.
Thus, if $z\in \ell_{x,p,L}$ satisfies \eqref{eqn:z1in} it follows 
from \eqref{eqn:zkineq} that $z\in D^{o},$ and so we conclude that 
the length of $\ell_{x,p,L}\cap D^{o}$ must
be at least $\frac{L}{8 \sqrt{d}}$. (Observe that if we would not 
have required $\ell_{x,p,L}$ to hit ``well inside'' the right boundary
of $D^o,$ i.e.~$R^{-16}(D^o),$ 
then $\ell_{x,p,L}$ could potentially have missed most or all
of $D^o.$) From this it then follows similarly to \eqref{eqn:TDbound}
that we may take
\begin{equation} \label{eqn:Nlowerbound}
N\geq \frac{1}{12}\frac{L}{8 \sqrt{d}}-1
=\frac{L}{96 \sqrt{d}}-1
\geq \frac{L}{1000 \sqrt{d}},
\end{equation}
which holds since we assume that $L>200\sqrt{d}.$
Informally, this simply means that one can find a string of order 
$L$ balls within the target $\CT(\ell_{x,p,L},D^o)$
such that the distance between consecutive centers of these balls 
is always 12. This is our second collection of balls.

Our next step is to prove that 
\[
\{\ell_{y,q,L}: B(t_i,2) 
\xleftrightarrow{\ell_{y,q,L}} B(z,2)\}
\cap 
\{\ell_{y,q,L}: B(t_j,2) 
\xleftrightarrow{\ell_{y,q,L}} B(w,2)\}=\emptyset
\]
whenever either $i\neq j,$ or $z,w\in \BT(D^{(0,2)})$ is 
such that $z\neq w,$ or both.
We will show this for $i\neq j$ and for $z=w.$ All other cases
follow in the same way. Therefore, assume that 
$\ell_{y,q,L}\cap B(z,2)\neq \emptyset$ and that 
$\ell_{y,q,L}\cap B(t_i,2)\neq \emptyset$.  We claim that 
$\ell_{y,q,L}\cap B(t_j,2) =\emptyset$ for $j\neq i.$
The key to showing this claim lies in analyzing the orientation
vectors $p,q$ and then using Lemma \ref{lemma:t1tau1} (see also the 
remark after the statement of that lemma). In order to 
use Lemma \ref{lemma:t1tau1} we need to prove that 
$|\langle p,q\rangle|\leq \frac{1}{\sqrt{2}}.$ We will do this by showing 
that $|p_k|$ is small for $k=2,\ldots,d$ and that 
$|q_k|$ is small for $k=1,3,\ldots,d.$

Recall that
$\ell_{x,p,L}\cap B(\gamma_x,2)\neq \emptyset$ and 
$\ell_{x,p,L}\cap B(\zeta_x,2)\neq \emptyset$ 
where $\gamma_x\in D^{(-2,0)}$ and $\zeta_x\in \BR(D^o).$ Therefore,
we can write 
\[
p=\frac{\gamma-\zeta}{\Vert \gamma-\zeta\Vert }
\]
for some $\gamma\in B(\gamma_x,2)$ and $\zeta\in B(\zeta_x,2)$. 
If we let $\xi=(\xi_1,\ldots,\xi_d)=\gamma-\zeta$ we then see that 
\[
2 \frac{L}{4}-4\leq |\xi_1|\leq 2 \frac{L}{4}+\frac{L}{8\sqrt{d}}+4
\textrm{ and that }
|\xi_k|\leq \frac{L}{8\sqrt{d}}-12,
\]
for $k=2,\ldots,d.$ It follows that $\Vert \xi\Vert \geq \frac{L}{2}-4$
and so we conclude that 
\[
|p_k|=\frac{|\xi_k|}{\Vert \xi\Vert }
\leq \frac{\frac{L}{8\sqrt{d}}-12}{\frac{L}{2}-4}
\leq \frac{1}{4\sqrt{d}},
\]
for $k=2,\ldots,d$ and every $d\geq 2$
(where of course $p=(p_1,\ldots,p_d)$). 

Again by assumption, 
$\ell_{y,q,L}\cap B(t_i,2)\neq \emptyset$  
and $\ell_{y,q,L}\cap B(z,2)\neq \emptyset$ where $z\in \BT(D^{(0,2)}).$
Therefore, we can write
\[
q=\frac{\tilde{\gamma}-\tilde{\zeta}}{\Vert \tilde{\gamma}-\tilde{\zeta}\Vert }
\]
for some $\tilde{\gamma}\in B(t_i,2)$ and 
$\tilde{\zeta}\in B(z,2).$ 
As for $p,$ it follows here that 
\[
|q_k|\leq \frac{1}{4\sqrt{d}},
\]
for $k=1,3,\ldots,d$ (since $q$ points almost vertically and $p$ points
almost horizontally). Therefore we can conclude that 
\begin{eqnarray*}
\lefteqn{|\langle p,q \rangle|
\leq |p_1| \cdot |q_1|+\cdots+|p_d| \cdot |q_d|}\\
& & \leq |q_1|+|p_2|+|p_3| \cdot |q_3|+\cdots+|p_d| \cdot |q_d| \\
& & \leq \frac{1}{4\sqrt{d}}+\frac{1}{4\sqrt{d}}+(d-2)\frac{1}{16 d}
\leq \frac{1}{\sqrt{2}},
\end{eqnarray*}
for every $d\geq 2.$ It follows from Lemma \ref{lemma:t1tau1} and the 
remark thereafter that 
$\ell_{y,q,L}\cap B(t_j,2)= \emptyset$ since 
$|t_i-t_j|\geq 12$.

We can now conclude from Lemma \ref{lemma:twoballmeasure2} that 
\begin{eqnarray} \label{eqn:muNTineq}
\lefteqn{\mu_\lambda\left((y,q): y\in D^{(0,1)}, 
\CT(\ell_{x,p,L},D^o)^{+1} \xleftrightarrow{\ell_{y,q,L}}
B(z,2) \textrm{ for some } z\in \BT(D^{(0,2)})\right)}\\
& & \geq 
\mu_\lambda\Big((y,q): y\in D^{(0,1)}, B(t_i,2)
\xleftrightarrow{\ell_{y,q,L}}B(z,2), \nonumber \\
& & \hspace{52mm}
\textrm{ for some } i=1,\ldots,N \textrm{ and } z\in \BT(D^{(0,2)})\Big)
\nonumber \\
& & = \sum_{i=1}^N \sum_{z\in \BT(D^{(0,2)})} 
\mu_\lambda\left((y,q): y\in D^{(0,1)},
B(t_i,2) \xleftrightarrow{\ell_{y,q,L}} B(z,2)\right) \nonumber\\
& & \geq N |\BT(D^{(0,2)})|\lambda \delta  c_d L^{-d+2} \nonumber
\end{eqnarray}
where the disjointness was used in the first equality and 
Lemma \ref{lemma:twoballmeasure2} in the second inequality. 
It follows from \eqref{eqn:TDbound} that 
\[
|\BT(D^{(0,2)})|\geq \left(\frac{L}{96\sqrt{d}}-4\right)^{d-1}
\geq \left(\frac{L}{1000\sqrt{d}}\right)^{d-1},
\]
since $L>200\sqrt{d.}$ Furthermore, by also using \eqref{eqn:Nlowerbound}
we conclude that 
\begin{eqnarray} \label{eqn:NTineq}
\lefteqn{N |\BT(D^{(0,2)})|\lambda \delta  c_d L^{-d+2}}\\
& & \geq \frac{L}{1000\sqrt{d}} \left(\frac{L}{1000\sqrt{d}}\right)^{d-1}
\lambda \delta  c_d L^{-d+2}
=\lambda \delta \frac{c_d}{\left(1000\sqrt{d}\right)^d}L^2
=\lambda \delta c_d' L^2. \nonumber
\end{eqnarray} 
Combining \eqref{eqn:muNTineq} and \eqref{eqn:NTineq}, we then see that 
\[
\mu_\lambda\left((y,q): y\in D^{(0,1)}, 
\CT(\ell_{x,p,L},D^o)^{+1} \xleftrightarrow{\ell_{y,q,L}}
B(z,2) \textrm{ for some } z\in \BT(D^{(0,2)})\right)
\geq \lambda \delta c_d'L^2.
\]
Therefore, 
\begin{eqnarray*}
\lefteqn{\BP\left(\exists (y,q)\in \Pi^\lambda, z\in \BT(D^{(0,2)}): 
y\in D^{(0,1)}, 
\CT(\ell_{x,p,L},D^o)^{+1} \xleftrightarrow{\ell_{y,q,L}}
B(z,2)\right)}\\
& & = 1-\exp\Big(-\mu_\lambda\Big((y,q): y\in D^{(0,1)}, \\
& & \hspace{40mm}
\CT(\ell_{x,p,L},D^o)^{+1} \xleftrightarrow{\ell_{y,q,L}}
B(z,2) \textrm{ for some } z\in \BT(D^{(0,2)})\Big)\Big) \\
& & \geq 1-e^{-\lambda \delta c_d'L^2}
\geq 1-\frac{1}{c_d'\lambda \delta L^2},
\end{eqnarray*}
where we used that $e^{-x}\leq x^{-1}$ for every $x>0.$
\end{proof}

We are now ready to prove the upper bound of Theorem \ref{thm:uniform2}.

\begin{proof}[Proof of upper bound of Theorem \ref{thm:uniform2}]
Consider the bond version of the oriented percolation model described
at the beginning of this section and recall the notation 
$(A_n)_{n\geq 0}.$ Recall also that we will here work on the lattice 
$2\BH$ so that $A_n \subset \{u\in \BH: u_1=2n\}.$ In our 
construction below, we will consider a sequence of random sets 
$(E_n)_{n\geq 0}$ defined by letting 
$u\in E_n$ if there exists a ``good'' path of sticks connecting
a base stick $S_{x,p,L}$ to the box $D^u.$ Our coupling will yield
$A_n \subset E_n$ for every $n\geq 1,$ where the value of $\alpha$ for 
the oriented percolation model will depend on $\lambda.$ We will then 
show that for $\lambda$ larger than the upper bound of the statement of 
this theorem, we will have that $\alpha\geq 0.81$ and therefore  
\eqref{eqn:OPsurvives} is satisfied. This then shows that there exists 
an unbounded connected component in $\CC(\Pi^\lambda)$ 
with positive probability. We choose to work with $\alpha\geq 0.81$
rather than $\alpha>2/3$ out of convenience, as making a more
optimal choice of $\alpha$ would not affect the quality of our bound
in any meaningful way.

We will now fix 
\begin{equation} \label{eqn:lambdafix}
\lambda> \frac{10}{9\delta c_d'}L^{-2}
=\frac{10 \left(1000 \sqrt{d}\right)^d}{9\delta c_d}L^{-2}
=\frac{10 \left(1000 \sqrt{d}\right)^d \sqrt{d}\Gamma(2d-1)}
{9\delta 2^{5(d-2)}\pi^{d/2-2}\Gamma(d/2)^3}L^{-2},
\end{equation}
so that $\lambda$ is above $1/2$ of the upper bound in the statement.
It will be convenient to use two independent Poisson point processes 
$\Pi^{\lambda}_1$ and $\Pi^{\lambda}_2,$ defined on the same 
probability space and with the same distribution as $\Pi^\lambda.$
Clearly, if $\Pi^{\lambda}_1$ and $\Pi^{\lambda}_2$ 
are independent, $\Pi^{\lambda}_1+\Pi^{\lambda}_2$ is equal 
to $\Pi^{2\lambda}$ in distribution (this is why the right-hand side 
of \eqref{eqn:lambdafix} equals $1/2$ of the upper bound of the 
statement).

We will proceed with our construction below in steps. The general idea 
is illustrated in Figure \ref{fig:OP}, and it may be useful to consult 
this when reading what follows. 
\begin{figure} 
\includegraphics[clip=true, trim=25 10 25 85, width=16.5cm,page=1]{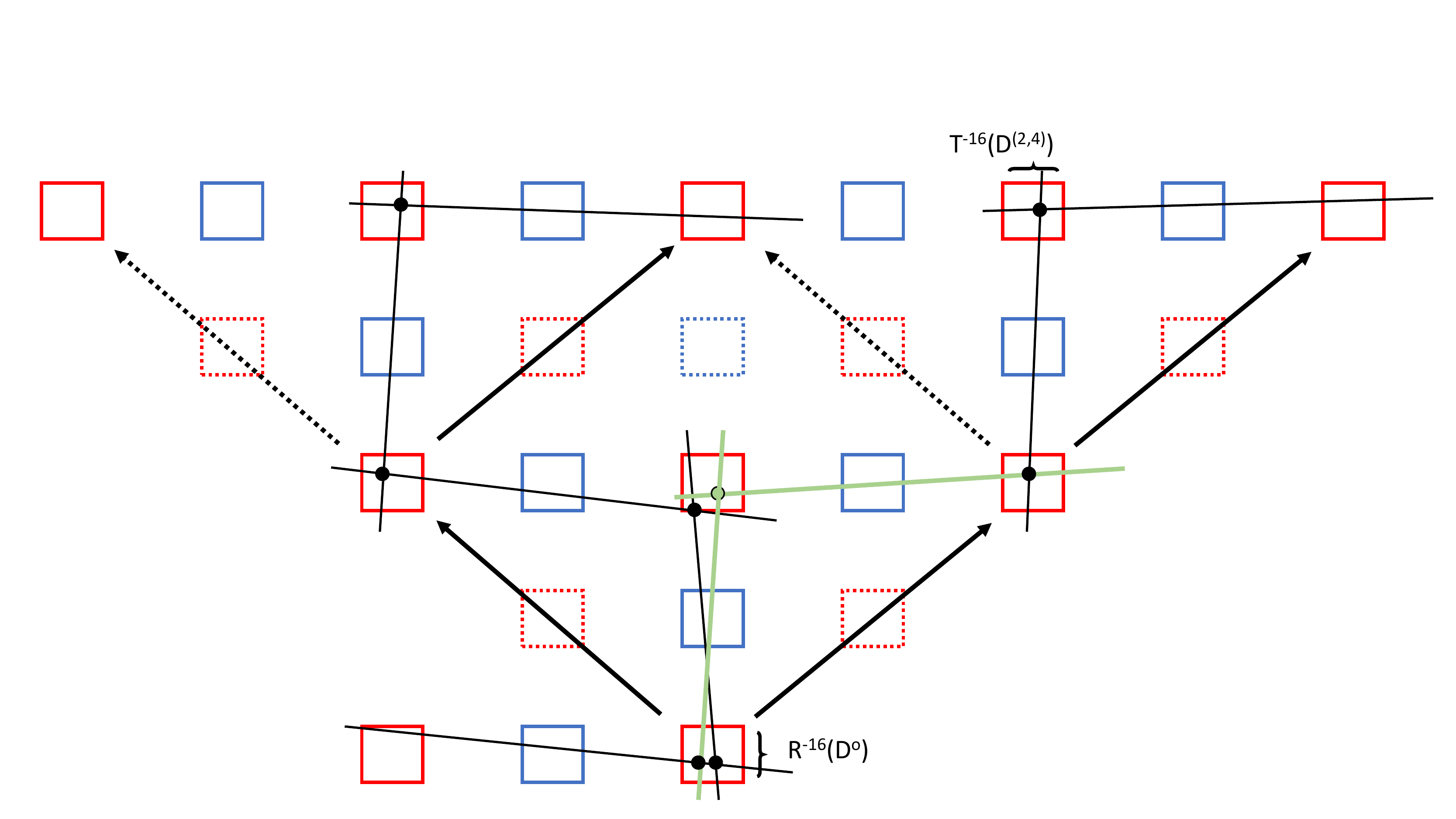} 
\caption{
An illustration of the coupling with oriented percolation. The
red boxes are target boxes, while the blue boxes are center boxes 
(i.e.~where the centers of the line segments belong). At the bottom is 
a line 
segment with 
center in $D^{(-1,0)}$ and which hits the target face $R^{-16}(D^o).$ 
The two sticks that are green (the color is only there 
for explanatory purposes) are the ones found during Step 1-a and 1-b,
and together they constitute a good right-oriented path from 
$\CT(\ell_{x,p,L},D^o)$ to $D^{(2,2)}$. We can also see a good 
left-oriented path from $\CT(\ell_{x,p,L},D^o)$ to $D^{(-2,2)}.$ 
These good paths induce connections from $o$
to $(-2,2)$ and $(2,2)$ in the oriented percolation model which are 
indicated by solid arrows. We also see connections from 
$(2,2)$ to $(4,4)$ and from $(-2,2)$ to $(0,4)$. However there are 
no connection from $(2,2)$ to $(0,4)$ nor from $(-2,2)$ to $(-4,4)$ 
(indicated by dashed, i.e.~missing arrows). Observe the line centered 
in $D^{(-1,2)},$ hitting the target in $D^{(0,2)}$ (and then proceeding 
to $D^{(-2,2)}$) but which does not hit $R^{-16}(D^{(0,2)})$.
Note however that (as required) it does hit $L^{-16}(D^{(-2,2)})$. Lastly, 
the dashed boxes are not used in the construction.} \label{fig:OP}
\end{figure}
In Step 0, observe that the event 
that there exists some 
\[
(x,p)\in \Pi^{\lambda}_1 \textrm{ such that } 
D^{(-2,0),+2} \stackrel{\ell_{x,p,L}}{\longleftrightarrow} R^{-16}(D^{o}),
\]
has positive probability for every $\lambda>0.$ We therefore condition 
on this event, and let $E_0=\{o\}.$ Observe that $A_0=E_0.$

Next, for Step 1-a we observe that by Lemma \ref{lemma:lineface} we 
have that 
\begin{align*}
\BP\left(\exists (y_v,q_v)\in \Pi^{\lambda}_1,
\zeta_{y_v}\in \BT(D^{(0,2)}): y_v\in D^{(0,1)},
\CT(\ell_{x,p,L},D^o)^{+1} \xleftrightarrow{\ell_{y_v,q_v,L}}
B(\zeta_{y_v},2)\right) \\
\geq 1-\frac{1}{\lambda \delta c_d' L^2},
\end{align*}
where $c_d'$ is as in that lemma. Using 
obvious notation we then let 
$
\CT(\ell_{y_v,q_v,L},D^{(0,2)})
$
denote the new, almost-vertical target provided by $\ell_{y_v,q_v,L}.$
Conditioned on the existence of such a line $\ell_{y_v,q_v,L}$ 
we can again use Lemma \ref{lemma:lineface} along with 
rotational invariance to see that
\begin{align*}
\BP\left(\exists (y_h,q_h)\in \Pi^{\lambda}_1,
\zeta_{y_h}\in \BR(D^{(2,2)}): 
y_h\in D^{(1,2)}, \CT(\ell_{y_v,q_v,L},D^{(0,2)})^{+1} 
\xleftrightarrow{\ell_{y_h,q_h,L}}
B(\zeta_{y_h},2)\right) \\
\geq 1-\frac{1}{\lambda \delta c_d' L^2},
\end{align*}
which is Step 1-b.
Since Step 1-a concerns line segments with centers
$y_v\in D^{(0,1)}$ and Step 1-b concerns line segments with 
centers $y_h\in D^{(1,2)},$ we conclude that the probability that 
there exists a ``right-oriented'' (see the green lines 
in Figure \ref{fig:OP}) path from $\CT(\ell_{x,p,L})$ to some 
$B(\zeta_{y_h},2)$ where $\zeta_{y_h}\in \BR(D^{(2,2)}),$
is at least 
\begin{equation} \label{eqn:squarealpha}
\left(1-\frac{1}{\lambda \delta c_d'L^2}\right)^2
\geq \left(1-\frac{1}{ \frac{10}{9\delta c_d'}L^{-2} c_d' \delta L^2}\right)^2
=\left(1-\frac{9}{10}\right)^2=0.81, 
\end{equation}
by using \eqref{eqn:lambdafix}. Note that the corresponding 
sticks $S_{x,p,L},S_{y_v,q_v,L}$ and $S_{y_h,q_h,L}$ form a connected 
path and we call such a path a {\em good} path.
Note also that the line $y_{h,q,L}$ 
presents a new target $\CT(\ell_{y_h,q,L},D^{(2,2)})$ from which the 
construction can proceed. If a good right-oriented path as described
exists, we let $(2,2)\in E_1$ and because of \eqref{eqn:squarealpha}
we can couple this with $A_1$ so that $(2,2)\in A_1$ with probability
$0.81.$

In the same way, we explore $\Pi^{\lambda}_2$ and attempt to find a left 
oriented good path from $D^o$ to $D^{(-2,2)}.$ Indeed, as above, the 
probability that there exists a good path within 
$\Pi^{\lambda}_2$ from $\CT(\ell_{x,p,L},D^o)$ via $D^{(0,2)}$  
to some $B(\zeta_h',2)$ where $\zeta_h'\in \BL(D^{(-2,2)}),$ and 
where the last line $y'_{h,q,L}$ presents 
a new target $\CT(\ell_{y'_h,q,L},D^{(-2,2)})$ is by symmetry 
at least $\alpha.$ We say that $(-2,2)\in E_1$ if such a left 
oriented path exists. Furthermore, since $\Pi^{\lambda}_1$ and 
$\Pi^{\lambda}_2$ are independent,
the existence of the right-oriented and the left-oriented paths are 
also independent, and so the events $\{u\in E_1\}$ for 
$u\in \{(-2,2),(2,2)\}$ are conditionally 
independent conditioned on the existence of $\CT(\ell_{x,p,L},D^o).$
Using \eqref{eqn:squarealpha}, we can couple 
this described procedure with the oriented
percolation model with parameter $\alpha=0.81$ in such a way that if there 
is an arrow from $o$ to $(-2,2)$ in the oriented percolation model, 
then there is a good path from 
$\CT(\ell_{x,p,L})$ to $D^{(-2,2)}$ in the stick process. In the same
way, if there is an arrow from $o$ to $(2,2)$ in the oriented percolation 
model, there is also a good path from $\CT(\ell_{x,p,L})$ to $D^{(2,2)}$ in 
the stick model. 
For our next step, we let $\CT_1=\{\CT^{u}\}_{u\in E_1}$ be 
the set of targets presented by the paths we found in Step 1.
Our coupling shows that $A_1 \subset E_1,$ and we note 
that there is a path of sticks from $D^o$ to $D^u$ for every $u\in E_1$ 
within $\CC(\Pi^{2\lambda})$ where 
$\Pi^{2\lambda}=\Pi^{\lambda}_1+\Pi^{\lambda}_2.$

Step 2 works in much the same way as Step 1, and so we condition on
$E_1$ and $\CT_1$ from Step 1. Given a target $\CT^u\in \CT_1,$ we
use $\Pi^{\lambda}_1$ to attempt to find a good right-oriented path 
from the target $\CT^u\in \CT_1$ to $D^{u+(2,2)}$ using two sticks
with centers in $D^{u+(0,1)}$ and $D^{u+(1,2)}$ respectively.
Furthermore, we use $\Pi^{\lambda}_2$ when attempting to find a left-oriented
path from $\CT^u\in \CT_1$ to $D^{u+(-2,2)}$ using two sticks
with centers in $D^{u+(0,1)}$ and $D^{u+(-1,2)}$ respectively.
Note that the target $\CT^u$ can be a result of Step 1 yielding a 
good path using $\Pi^{\lambda}_1,$ but that we attempt to find a 
good path to $D^{u+(-2,2)}$ using $\Pi^{\lambda}_2.$ This is not an 
issue as we in the end consider $\Pi^{\lambda}_1+\Pi^{\lambda}_2$
(or rather the union of all sticks associated to points 
from $\Pi^{\lambda}_1$ or $\Pi^{\lambda}_2$, i.e.~$\CC(\Pi^{2\lambda})$). 
We note further that 
if $E_1=\{(-2,2),(2,2)\},$ then there can be a good path both from 
$(-2,2)$ to $(0,4)$ and from $(2,2)$ to $(0,4).$ If both of these 
events occur, then there are two possible targets in $D^{(0,4)}$ 
for Step 3. In order to avoid ambiguities, we will in this and all 
similar cases always use the target $\CT^{(0,4)}$ provided by the 
configuration $\Pi^{\lambda}_1$ from the previous step (in this case 
Step 2). 
Since finding a good path always has probability at least $\alpha,$ 
we can couple $A_2,E_2$ such that $A_2 \subset E_2.$ As before, we 
let $\CT_2=\{\CT^u\}_{u\in E_2}$ be the set of targets provided by 
the good paths.

The general step is now clear. Given $E_n$ and the targets $\CT_n$ 
from Step $n,$ we look for good paths from $D^u$ to $D^{u+(2,2)}$ 
using $\Pi^{\lambda}_1,$ and from $D^u$ to $D^{u+(-2,2)}$ 
using $\Pi^{\lambda}_2.$ With this coupling we see that $A_n \subset E_n$
for every $n\geq 0,$ and since we chose $\alpha=0.81$ so that 
\eqref{eqn:OPsurvives} holds, we conclude that 
\[
\BP(E_n\neq \emptyset \ \ \forall n\geq 0)>0,
\]
for $\lambda$ as in \eqref{eqn:lambdafix}.
Clearly, 
if the event $\{E_n\neq \emptyset \ \ \forall n\geq 0\}$ occurs, then
$S_{x,p,L}$ belongs to an unbounded connected component, and so 
$\CC(\Pi^{2\lambda})$ percolates.  
\end{proof}

\subsection{The upper bound of Theorem \ref{thm:rigid2}} \label{sec:rigidupperbound}

We now turn to the proof of the upper bound of Theorem \ref{thm:rigid2}.
Again, we will couple our stick process with oriented percolation, but 
this time to the site percolation model on $\BH.$

\begin{proof}[Proof of upper bound of Theorem \ref{thm:rigid2}]
In the rigid case we have that $p=e_2$ with probability one, and so 
we will simply write $(x,e_2).$ 

Similar to the notation $D^u,$ for $u\in \BH,$ we let 
\[
B^u(1/2):=B((u_1,(L/2+2)u_2,0,\ldots,0),1/2),
\]
so that $B^u(1/2)\subset \BR^d$ is a closed ball of radius 1/2 
corresponding to the point $u\in \BH.$ 

For $u\in \BH$ such that $u_2\geq 1,$ 
we now consider the event
\begin{eqnarray*}
\lefteqn{G_{u}=\{\omega\in \Omega: \exists (x,e_2)\in \omega 
\textrm{ such that } B^u(1/2)\subset 
S_{x,e_2,L},}\\
& & B^{u+(-1,-1)}(1/2)\cap S_{x,e_2,L}\neq \emptyset 
\textrm{ and } B^{u+(1,-1)}(1/2)\cap S_{x,e_2,L}\neq \emptyset\}.
\end{eqnarray*}
The event $G_u$ implies that the ball $B^u(1/2)$
is completely covered by a stick, while the same stick intersects the two 
balls to the lower left and right (see Figure \ref{fig:OPrigid1}). 
Furthermore, we note that in order for a stick $S_{x,e_2,L}$ to contain 
both $B^u(1/2)$ and $B^{u+(-1,-1)}(1/2),$ it must be that 
$x_2=u_2-1/2.$ This event has measure 0.
Furthermore, since the centers of $B^{u}(1/2)$ 
and $B^{u+(0,2)}(1/2)$ are at distance $L+4,$ no stick can contain 
both of these balls. It follows that outside of an event of 
measure 0, a stick can only contain one ball, and therefore 
the events $\{G_u\}_{u\in \BH, u_2\geq 1}$ are independent. Furthermore, 
we note that if $\Pi^\lambda \in G_u \cap G_{u+(-1,-1)},$ then the 
corresponding sticks which contain $B^u(1/2)$ and $B^{u+(-1,-1)}(1/2)$ 
respectively must touch, and that they therefore belong to the 
same connected component of $\CC(\Pi^\lambda)$ 
(see also Figure \ref{fig:OPrigid2}). 
We can use these properties to couple the stick process with the 
site percolation variant of the oriented percolation model as 
we now explain.  

\begin{figure}
\centering
\begin{subfigure}{.5\textwidth}
  \centering
  \includegraphics[clip=true, trim=180 100 480 55, width=0.7\linewidth,
  page=1]{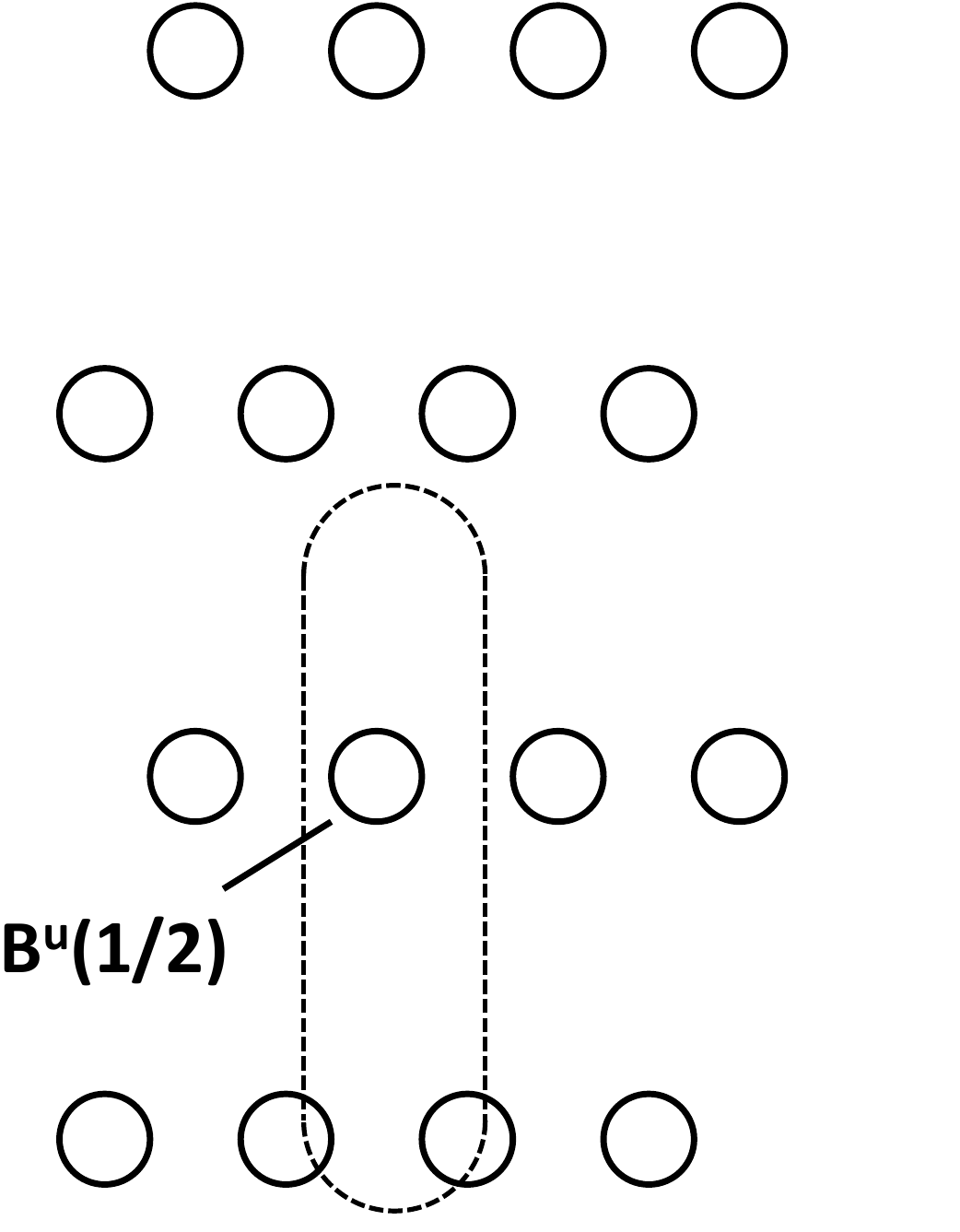}
  \caption{}
  \label{fig:OPrigid1}
\end{subfigure}%
\begin{subfigure}{.5\textwidth}
  \centering
  \includegraphics[clip=true, trim=180 100 480 55, width=0.7\linewidth,page=2]{OPrigid.pdf}
  \caption{}
  \label{fig:OPrigid2}
\end{subfigure}
\caption{In (a) the event $G_u$ is illustrated. The ball corresponding 
to $u\in \BH,$ i.e.~$B^u(1/2),$ is 
completely covered by the stick (dashed). Furthermore, the stick
intersects the two balls $B^{u+(-1,-1)}(1/2)$ and $B^{u+(1,-1)}(1/2)$
one row below. In (b) the events 
$G_u,G_{u+(1,1)}$ and $G_{u+(2,2)}$ all occur. We see that the 
corresponding sticks are part of the same connected component.}
\label{fig:OPrigid}
\end{figure}

First, it is clear that 
\[
\BP(\exists (x,e_2)\in \Pi^\lambda: B^o(1/2) \subset S_{x,e_2,L})>0,
\]
and if the event in this probability occurs, we set $E_0=A_0=\{o\}.$ 
Then, conditioned on $E_0=\{o\},$ we let 
\[
E_1=\{u\in \BH: \Pi^\lambda \in G_u\} \cap 
\{u\in \{(-1,1), (1,1)\}\},
\]
and we see that $\BP(u\in E_1 |E_0)=\BP(G_u)$ and that the
events $(-1,1)\in E_1$ and $(1,1)\in E_1$ are conditionally independent.
In general we let 
\[
E_n=\{u\in \BH: \Pi^\lambda \in G_u\} \cap 
\{u\in \{u'+(1,-1), u'+(1,1)\},
\textrm{ where } u'\in E_{n-1}\},
\]
and again we see that $\BP(u\in E_n |E_{n-1})=\BP(G_u)$ for 
any $u$ such that $u\in \{u'+(-1,1), u'+(1,1)\},
\textrm{ where } u'\in E_{n-1}.$ Conditional independence is 
also clear. We see that this is indeed the site percolation 
version of the oriented percolation model, since both probabilities
in the two first rows of \eqref{eqn:OPdef} are $\BP(G_u)$ in our case. 
Since this version survives with positive probability if $\alpha>3/4$
(recall the discussion at the start of Section \ref{sec:upperbounds})
we need to show that $\BP(G_u)> 3/4$ whenever 
\begin{equation} \label{eqn:lambdaexceeds}
\lambda>4 \frac{2^d \Gamma((d+1)/2)}{\pi^{d/2-1}}L^{-1}.
\end{equation}

In order to bound $\BP(G_u)$ we consider, without loss of generality, 
the case where $u=(1,1).$ Recall that the center of $B^{(1,1)}(1/2)$
is at the point $(1,(L/2+2),0,\ldots,0)$, and
observe that in order for $B^{(1,1)}(1/2)\subset S_{x,e_2,L}$ 
it is sufficient that 
$\Vert (x_1,x_3,\ldots,x_d) -(1,0,\ldots,0)\Vert 
=\Vert (x_1-1,x_3,\ldots,x_d) \Vert \leq 1/2$ and that 
$x_2 \in [2,L+2].$ Next, in order for $G_u$ to occur, we 
must also have that $G^o(1/2)\cap S_{x,e_2,L} \neq \emptyset.$ 
For this, it is sufficient that 
$\Vert (x_1,x_3,\ldots,x_d) \Vert\leq 3/2$ and that $x_2\in[-L/2,L/2].$
Similarly, in order for $G^{(2,0)}(1/2)\cap S_{x,e_2,L} \neq \emptyset$ 
it is sufficient that 
$\Vert (x_1,x_3,\ldots,x_d) -(2,0,\ldots,0)\Vert
=\Vert (x_1-2,x_3,\ldots,x_d)\Vert\leq 3/2$ 
and that $x_2\in[-L/2,L/2].$
We then observe that all of these three conditions are satisfied whenever
\[
\Vert (x_1-1,x_3,\ldots,x_d) \Vert \leq 1/2 
\textrm{ and } x_2\in[2,L/2].
\]
Therefore (recall the notation $B_{d-1}(x,r)$ from 
Section \ref{sec:model}), 
\begin{eqnarray*}
\lefteqn{\mu_\lambda((x,e_2):B^{(1,1)}(1/2)\subset S_{x,e_2,L},
B^o(1/2)\cap S_{x,e_2,L}\neq \emptyset }\\
& & \hspace{50mm}
\textrm{ and } B^{(2,0)}(1/2)\cap S_{x,e_2,L}\neq \emptyset)\\
& & = \lambda \int_{x \in \BR^d} 
I(\Vert (x_1-1,x_3,\ldots,x_d) \Vert \leq 1/2 
\textrm{ and } x_2\in[2,L/2]) \d x \\ 
& & =\lambda (L/2-2) \Vol(B_{d-1}(o,1/2))
=\lambda (L/2-2)\frac{\pi^{(d-1)/2}}{2^{d-1}\Gamma((d+1)/2)} \\
& & > \lambda L\frac{\pi^{d/2-1}}{2^{d}\Gamma((d+1)/2)},
\end{eqnarray*}
where the last inequality holds for all $L>10$.

It follows similarly to the end of the proof 
of Lemma \ref{lemma:lineface} that 
\[
\BP(G_u)>
1-\exp\left(-\lambda L\frac{\pi^{d/2-1}}{2^{d}\Gamma((d+1)/2)}\right)
\geq 1-\frac{2^{d}\Gamma((d+1)/2)}{\lambda L \pi^{d/2-1}},
\]
which is larger than $3/4$ whenever \eqref{eqn:lambdaexceeds} holds.
We conclude that for such values of $\lambda$ there is a positive 
probability that $B^o(1/2)$ belongs to an unbounded connected component
of $\CC(\Pi^\lambda),$ and therefore $\lambda \geq \lambda_{c,r}.$
\end{proof}

\begin{appendices}
\section{}\label{app:main}
The purpose of this appendix is to provide detailed proofs
of the key lemmas used in Sections \ref{sec:lowerbounds} and 
\ref{sec:upperbounds}, i.e.~Lemmas
\ref{lemma:msrenumberofballs}, \ref{lemma:twoballmeasure2} and 
\ref{lemma:t1tau1}. We will start with Lemma 
\ref{lemma:msrenumberofballs} as this is easy.  
The proofs of the other two lemmas go through several steps, and we 
therefore put these into separate subsections.

\subsection{Measure estimates} \label{app:4.1}

Here, and in the rest of the appendix, we will let $I(\cdot)$ denote 
an indicator function. \\

\begin{proof}[Proof of Lemma \ref{lemma:msrenumberofballs}]
We have that 
\begin{eqnarray*}
\lefteqn{\mu_\lambda((x,p)\in \BR^d \times \BS: 
\ell_{x,p,L}\cap B(o,\rho)\neq \emptyset)}\\
& & =\int_{\BS}\int_{\BR^d}  I(\ell_{x,p,L}\cap B(o,\rho)) \d x 
\Theta(\d p)
=\int_{\BS}\int_{\BR^d}  I(\ell_{o,p,L}\cap B(x,\rho)) \d x 
\Theta(\d p) \\
& & =\int_{\BS} \Vol(S_{o,p,L}(\rho))\Theta(\d p)
=\Vol(S_{o,p,L}(\rho)) \Theta(\BS)=\Vol(S_{o,p,L}(\rho)),
\end{eqnarray*}
where the second equality follows by translation invariance.
By construction, the volume of $S_{o,p,L}(\rho)$ equals 
the volume of a $(d-1)$-dimensional ball of radius $\rho$ 
times $L,$ plus the volume
of the  two ``tips'' which together equals that of a $d$-dimensional 
ball of radius $\rho$. That is,
\[
\Vol(S_{o,p,L}(\rho))
=L\Vol(B_{d-1}(o,\rho))+\Vol(B_{d}(o,\rho))
=L\frac{\pi^{(d-1)/2}}{\Gamma((d+1)/2)}\rho^{d-1}
+\frac{\pi^{d/2}}{\Gamma(d/2+1)}\rho^d,
\]
as desired.
\end{proof}

\subsection{Proof of Lemma \ref{lemma:twoballmeasure2}}
\label{app:5.1}
In order to prove Lemma \ref{lemma:twoballmeasure2}, we will go through 
two steps. The first of these is the following lemma which
provides a lower bound 
on the measure of line segments that hit a given ball of 
radius $\rho.$ 

\begin{lemma}[\bf Lower bound on line segment-ball hitting
measure]\label{lemma:xtoball}
Let $d\geq 2,$ $L>0$ and  $0<\rho<L/8.$ Then, for any $x\in \BR^d$ with  
$\rho<\Vert x\Vert < L/2-\rho, $ we have that 
\[
\int_{\BS} I(\ell_{x,p,L}\cap B(o,\rho)\neq \emptyset)\CH(\d p)
\geq \frac{\Gamma(d/2)}{\sqrt{\pi}\Gamma((d+1)/2)}
\frac{\rho^{d-1}}{\Vert x\Vert ^{d-1}}.
\]
\end{lemma}
\noindent
{\bf Remark:} We will only use Lemma \ref{lemma:xtoball} for 
$0<\rho\leq 2.$
However, we chose to state it more generally as it is essentially 
a more exact 
version of the lower bound of Lemma 3.1 of \cite{TW_12}, which has 
been used many times (for instance in \cite{BT_16}).

We also note that it is possible to derive a weaker version of Lemma 
\ref{lemma:msrenumberofballs} from Lemma \ref{lemma:xtoball} by 
integrating over a suitable subset of $\BR^d.$

\medskip

\begin{proof}
By rotational invariance of the model, we may
without loss of generality assume that $x=(x_1,0,\ldots,0)$
where $x_1 >\rho.$ As in Section \ref{sec:mainupperbound}
we let
\[
\ell_{x,p,\infty}=\{x+tp: -\infty<t<\infty\},
\]
and $p$ is as usual a vector such that $\Vert p\Vert =1.$ Then, 
consider the distance between a point $\ell_{x,p,\infty}(t)$ on the 
line $\ell_{x,p,\infty}$ and the origin $o$ 
\[
\Vert \ell_{x,p,\infty}(t)-o \Vert^2
=\Vert x+tp\Vert ^2=\Vert x\Vert ^2+t^2\Vert p\Vert ^2+2t\langle x,p \rangle
=\Vert x\Vert ^2+t^2+2t\langle x,p \rangle
\]
which is minimized when $t=-\langle x,p\rangle,$ so that 
\[
\dist(\ell_{x,p,\infty},o)^2
=\Vert x\Vert ^2+\langle x,p\rangle^2-2\langle x,p\rangle^2
=\Vert x\Vert ^2-\langle x,p\rangle^2.
\]
Using that $x=(x_1,0,\ldots,0)$ we then see that 
\begin{equation} \label{eqn:distrho}
\dist(\ell_{x,p,\infty},o)^2\leq \rho^2 \Leftrightarrow
x_1^2-x_1^2 p_1^2 \leq \rho^2 \Leftrightarrow
p_1^2 \geq \frac{x_1^2-\rho^2}{x_1^2}.
\end{equation}
We note that if $x_1\leq \rho,$ then \eqref{eqn:distrho} holds for 
any $p$ as indeed it should since $x$ is then inside of the ball 
$B(o,\rho)$. This is why we include $\Vert x \Vert >\rho$ in the 
assumption. Furthermore, the line segment $\ell_{x,p,L}$ has finite 
length $L,$ 
which is why we need to assume that $x_1=\Vert x \Vert< L/2-\rho$
in order for $\ell_{x,p,L}$ to reach the ball $B(o,\rho)$.
To see that this suffices, we observe that
\[
\Vert \ell_{x,p,\infty}(-\langle x,p\rangle)-x_1 \Vert
\leq \dist(\ell_{x,p,\infty},o)+x_1 
< L/2,
\]
whenever \eqref{eqn:distrho} holds (since $-\langle x,p\rangle$ minimized 
the distance between $\ell_{x,p,\infty}$ and $o$).
It follows that for every $\rho < x_1< L/2-\rho,$ we have that 
$\ell_{x,p,L}\cap B(o,\rho)\neq \emptyset$
if  and only if $p_1^2 \geq \frac{x_1^2-\rho^2}{x_1^2}.$
Thus, for fixed $\rho <x_1< L/2-\rho,$ we get that
\begin{eqnarray} \label{eqn:intequality}
\lefteqn{\int_{\BS}I(\ell_{x,p,L}\cap B(o,\rho)\neq \emptyset)\CH(\d p)}\\
& & =\int_{\BS} I(\dist(\ell_{x,p,\infty},o)^2\leq \rho^2)\CH(\d p)
=\int_{\BS} I\left(p_1^2\geq 1-\frac{\rho^2}{x_1^2}\right)\CH(\d p).
\nonumber
\end{eqnarray}
This is simply twice (by symmetry of $p$ and $-p$) the surface 
area (or rather the $(d-1)$-dimensional 
Hausdorff measure) of the spherical cap of height 
\[
h
=1-\sqrt{1-\frac{\rho^2}{x_1^2}},
\]
normalized (because of \eqref{eqn:norm}) by the surface area of 
$\BS,$ which is $\frac{2 \pi^{d/2}}{\Gamma(d/2)}.$
It is known (see \cite{L_11}), that the surface area of this spherical 
cap equals 
\[
\frac{\pi^{d/2}}{\Gamma(d/2)}
J_{2h-h^2}\left(\frac{d-1}{2},\frac{1}{2}\right)
=\frac{\pi^{d/2}}{\Gamma(d/2)}
J_{\rho^2/x_1^2}\left(\frac{d-1}{2},\frac{1}{2}\right),
\]
where $J$ is the regularized incomplete Beta function.
Furthermore,
\begin{eqnarray*}
\lefteqn{J_{\rho^2/x_1^2}\left(\frac{d-1}{2},\frac{1}{2}\right)}\\
& & =\frac{\int_0^{\rho^2/x_1^2} t^{(d-1)/2-1}(1-t)^{1/2-1}\d t}
{\int_0^1 t^{(d-1)/2-1}(1-t)^{1/2-1}\d t}
= \frac{\Gamma(d/2)}{\sqrt{\pi}\Gamma((d-1)/2)}
\int_0^{\rho^2/x_1^2} t^{(d-1)/2-1}(1-t)^{-1/2}\d t \\
& & \geq \frac{\Gamma(d/2)}{\sqrt{\pi}\Gamma((d-1)/2)}
\int_0^{\rho^2/x_1^2} t^{(d-1)/2-1}\d t 
=\frac{2\Gamma(d/2)}{(d-1)\sqrt{\pi}\Gamma((d-1)/2)}
\left[t^{(d-1)/2}\right]_0^{\rho^2/x_1^2} \\
& & =\frac{2\Gamma(d/2)}{(d-1)\sqrt{\pi}\Gamma((d-1)/2)}
\frac{\rho^{d-1}}{x_1^{d-1}}
=\frac{\Gamma(d/2)}{\sqrt{\pi}\Gamma((d+1)/2)}
\frac{\rho^{d-1}}{x_1^{d-1}},
\end{eqnarray*}
where the first equality is simply the definition of 
$J_{\rho^2/x_1^2}\left(\frac{d-1}{2},\frac{1}{2}\right),$ where 
the last equality uses the identity $z\Gamma(z)=\Gamma(1+z)$
with $z=(d-1)/2,$ and where we used that 
\[
\int_0^1 t^{(d-1)/2-1}(1-t)^{-1/2}dt
=\frac{\sqrt{\pi}\Gamma((d-1)/2)}{\Gamma(d/2)}
\]
for $d\geq 2.$ 
We then see that by \eqref{eqn:intequality} and 
\eqref{eqn:norm},
\begin{eqnarray} \label{eqn:msreJ}
\lefteqn{\int_{\BS}I(\ell_{x,p,L}\cap B(o,\rho)\neq \emptyset)\CH(\d p)
=\int_{\BS} I\left(p_1^2\geq 1-\frac{\rho^2}{x_1^2}\right)\CH(\d p)}\\
& & = 2\frac{\Gamma(d/2)}{2 \pi^{d/2}} \frac{\pi^{d/2}}{\Gamma(d/2)}
J_{\rho^2/x_1^2}\left(\frac{d-1}{2},\frac{1}{2}\right)
\geq \frac{\Gamma(d/2)}{\sqrt{\pi}\Gamma((d+1)/2)}
\frac{\rho^{d-1}}{x_1^{d-1}}, \nonumber
\end{eqnarray}
and so for general $\rho <\Vert x\Vert < L/2-\rho$ we get that
\[
\int_{\BS} I(\ell_{x,p,L}\cap B(o,\rho)\neq \emptyset) \CH(\d p)
\geq \frac{\Gamma(d/2)}{\sqrt{\pi}\Gamma((d+1)/2)}
\frac{\rho^{d-1}}{\Vert x \Vert^{d-1}}
\]
as required. 
\end{proof}

Our next step is to obtain a lower bound on the 
$\mu_\lambda$-measure of the set 
of points $(x,p)\in \BR^d \times \BS$ such that their corresponding 
line segments $\ell_{x,p,L}$ intersect two balls $B(\gamma,2)$ 
and $B(\zeta,2),$ placed so that their centers $\gamma,\zeta$ 
belong to the horizontal axis. To that end, we will need the following 
notation. Recall $x^r$ defined in \eqref{eqn:defxr} and let 
\[
\ell_{L/(32\sqrt{d})}:=\left\{x\in \BR^d: |x_1|
\leq \frac{L}{32\sqrt{d}} \textrm{ and } x_k=0 \textrm{ for }
k=2,\ldots,d \right\},
\]
$S_{L/(32\sqrt{d})}:=\ell_{L/(32\sqrt{d})}^{+2}$ and
\[
\Cyl_{L/(32\sqrt{d})}:=\left\{x\in \BR^d: |x_1|
\leq \frac{L}{32\sqrt{d}} \textrm{ and } |x^0|\leq 2 \textrm{ for }
k=2,\ldots,d\right\}.
\]
We see that $\ell_{L/(32\sqrt{d})}$ is a horizontal line segment of 
length $\frac{L}{16 \sqrt{d}}$ while $S_{L/(32\sqrt{d})}$ is the 
corresponding stick of radius 2 and $\Cyl_{L/(32\sqrt{d})}$ is the 
truncation of $S_{L/(32\sqrt{d})}$ where the tips have been removed.
In order to obtain the 
required lower bound, we will restrict our attention to 
$x\in S_{L/(32\sqrt{d})}.$ We remark that $S_{L/(32\sqrt{d})}$
will only serve as a set in which we are looking for centers $x$ 
of line segments $\ell_{x,p,L}.$ It just happens to be a stick of 
radius 2.

Recall the notation \eqref{eqn:conndef}. 
\begin{lemma}[\bf Measure of line segments connecting two balls] \label{lemma:twoballmeasure}
Let $d\geq 2,$ and let  
$\Theta(\d p)=\phi(p) \CH(\d p)$ where $\phi(p)$ satisfies
\eqref{eqn:phidelta} for some $\delta>0.$ Furthermore, let $L>32,$
and let $r_1,r_2\geq 0$ be such that 
$\frac{L}{8}<\min(r_1,r_2)
\leq \max(r_1,r_2)< \frac{L}{\sqrt{6}}.$
Then, 
\[
\mu_\lambda\left((x,p): x\in S_{L/(32\sqrt{d})}, 
B(-r_1e_1,2)
\xleftrightarrow{\ell_{x,p,L}}
B(r_2 e_1,2)\right)
\geq \lambda \delta  c_d L^{-d+2},
\]
where we may take 
\[
c_d=2^{5(d-2)} \pi^{d/2-2} \frac{1}{\sqrt{d}} 
\frac{\Gamma(d/2)^3}{\Gamma(2d-1)}.
\]
\end{lemma}
\noindent
{\bf Remarks:} It is easy to make the statement somewhat more general. In 
particular, one could consider more general widths of the cylinder 
and ease the requirements on $r_1,r_2.$ However, in contrast to 
Lemma \ref{lemma:xtoball} we do not anticipate a wider use of this lemma, 
and therefore we prefer to keep it as simple as possible.

In the statement we consider $(x,p)$ such that $x\in S_{L/(32\sqrt{d})},$
but the proof will restrict this further to $x \in \Cyl_{L/(32\sqrt{d})}.$
We chose to state it in the current way as this is how it will be used.

\medskip
\begin{proof}
We start by considering the case $x=o.$
Let $r=\max(r_1,r_2)$ and observe that by symmetry,
\[
\left\{B(-r_1e_1,2)
\xleftrightarrow{\ell_{o,p,L}}
B(r_2 e_1,2)\right\}
=\left\{B(-r e_1,2)
\xleftrightarrow{\ell_{o,p,L}}
B(r e_1,2) \right\},
\]
since here we used $x=o.$ 
Furthermore, we clearly have that 
\[
\left\{B(-r e_1,2)
\xleftrightarrow{\ell_{o,p,L}}
B(r e_1,2) \right\}
=\left\{\ell_{o,p,L} \cap B(r e_1,2) \neq \emptyset \right\},
\]
since any $\ell_{o,p,L}$ touching $B(r e_1,2)$ must also touch
$B(-r e_1,2)$ (again by symmetry). We therefore conclude that 
\begin{eqnarray} \label{eqn:twoballs}
\lefteqn{\int_{\BS} 
I \left(B(-r_1e_1,2)
\xleftrightarrow{\ell_{o,p,L}}
B(r_2 e_1,2)\right) \CH(\d p)}\\
& & =\int_{\BS} 
I \left(\ell_{o,p,L} \cap B(r e_1,2) \neq \emptyset \right) \CH(\d p)
=\int_{\BS} 
I \left(\ell_{r e_1,p,L} \cap B(o,2) \neq \emptyset \right) \CH(\d p) 
\nonumber\\
& & 
\geq \frac{\Gamma(d/2)}{\sqrt{\pi}\Gamma((d+1)/2)}
\frac{2^{d-1}}{r^{d-1}} 
=\frac{\Gamma(d/2)}{\sqrt{\pi}\Gamma((d+1)/2)}
 \left(\frac{2}{\max(r_1,r_2)}\right)^{d-1}, \nonumber
\end{eqnarray}
by translation invariance, and by
using Lemma \ref{lemma:xtoball} with $\rho=2,$ which we can since 
$L>16=8 \rho$ and 
\[
2<\frac{L}{8}<\Vert r e_1 \Vert< \frac{L}{\sqrt{6}}
\leq  \frac{L}{2} -2,
\]
by the assumption that $L>32.$

Next, we consider general $x\in \Cyl_{L/(32\sqrt{d})},$ and the idea 
is to reduce this
general case to the simpler first case where $x=o$ (see also Figure 
\ref{fig:BallsBa}).
\begin{figure} 
\includegraphics[clip=true, trim=0 160 0 160, width=15.5cm]{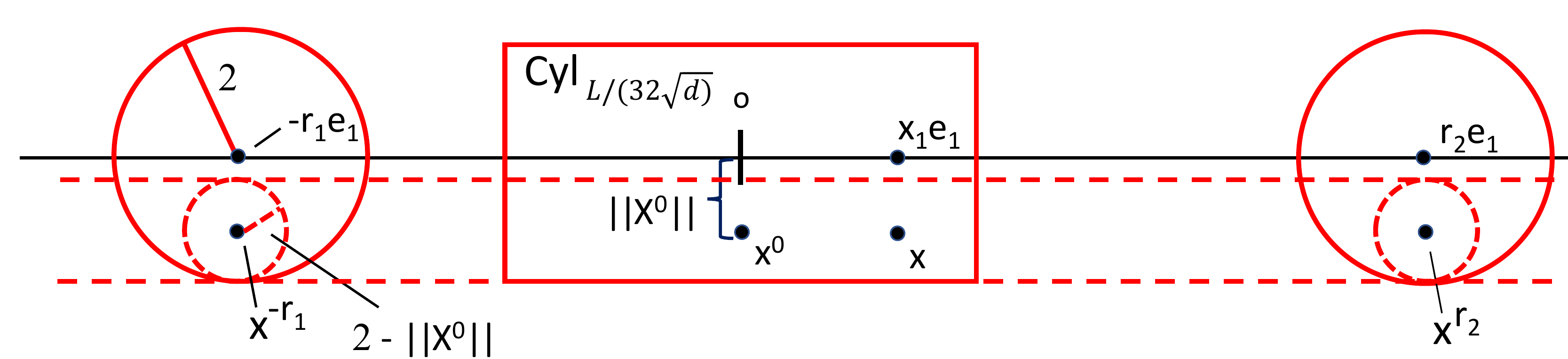} 
\caption{
An illustration of the argument leading up to \eqref{eqn:Ballinclusion}. 
The next step is to 
translate the picture so that $x$ is mapped to the origin $o.$} \label{fig:BallsBa}
\end{figure}
 As we will see, 
this can be done by translation and by replacing the balls 
$B(-r_1 e_1,2)$ with smaller ones (i.e.~$B(x^{-r_1},2-\Vert x^0\Vert )$)
whose centers are more conveniently placed. To this end,
observe that 
\[
B(x^{-r_1},2-\Vert x^0\Vert )\subset B(-r_1 e_1,2) 
\textrm{ and } 
B(x^{r_2},2-\Vert x^0\Vert )\subset B(r_2 e_1,2).
\] 
Indeed, let $y\in B(x^{-r_1},2-\Vert x^0\Vert )$ and note that  
\[
\Vert -r_1 e_1-y\Vert\leq \Vert -r_1 e_1-x^{-r_1}\Vert
+\Vert x^{-r_1}-y\Vert
=\Vert x^0\Vert +\Vert x^{-r_1}-y\Vert
\leq \Vert x^0\Vert +2-\Vert x^0\Vert =2.
\]
Therefore,
\begin{equation} \label{eqn:Ballinclusion}
\left\{B(x^{-r_1},2-\Vert x^0\Vert )
\xleftrightarrow{\ell_{x,p,L}}
B(x^{r_2},2-\Vert x^0\Vert )\right\}
\subset
\left\{B(-r_1 e_1,2)
\xleftrightarrow{\ell_{x,p,L}}
B(r_2 e_1,2)\right\},
\end{equation}
and so 
\begin{eqnarray} \label{eqn:ekorre}
\lefteqn{\int_{\BS} 
I \left(B(-r_1e_1,2)
\xleftrightarrow{\ell_{x,p,L}}
B(r_2 e_1,2)\right)\CH(\d p)}\\
& & \geq \int_{\BS} 
I \left(B(x^{-r_1},2-\Vert x^0\Vert )
\xleftrightarrow{\ell_{x,p,L}}
B(x^{r_2},2-\Vert x^0\Vert )\right) \CH(\d p) \nonumber \\ 
& & =\int_{\BS} 
I \left(B(x^{-r_1}-x,2-\Vert x^0\Vert )
\xleftrightarrow{\ell_{o,p,L}}
B(x^{r_2}-x,2-\Vert x^0\Vert )\right)\CH(\d p)\nonumber \\
& & =\int_{\BS} 
I \left(B(-(r_1+x_1) e_1,2-\Vert x^0\Vert )
\xleftrightarrow{\ell_{o,p,L}}
B((r_2-x_1) e_1,2-\Vert x^0\Vert )\right)\CH(\d p)\nonumber\\
& & \geq \frac{\Gamma(d/2)}{\sqrt{\pi}\Gamma((d+1)/2)}
 \left(\frac{2-\Vert x^0\Vert }{\max(r_1+x_1,r_2-x_1)}\right)^{d-1},
 \nonumber 
\end{eqnarray}
where the first equality follows by translation invariance
(translating $x$ to $o$), and the second inequality follows by using Lemma
\ref{lemma:xtoball} with $\rho=2-\Vert x^0\Vert.$
As before, we may use this lemma since by assumption
\[
\max(r_1+x_1,r_2-x_1) \geq \min (r_1,r_2)-|x_1|
>\frac{L}{8}-\frac{L}{32 \sqrt{d}}
>2\geq 2-\Vert x^0 \Vert =\rho,
\]
and 
\[
\max(r_1+x_1,r_2-x_1)\leq \max(r_1,r_2)+\frac{L}{32 \sqrt{d}}
<\frac{L}{\sqrt{6}}+\frac{L}{32 \sqrt{d}}
<\frac{L}{2}-2,
\]
since $L>32.$ Clearly we also have that 
$L>8 (2-\Vert x^0\Vert)=8\rho.$
Furthermore, since $\max(r_1+x_1,r_2-x_1)<L/2,$ it follows from 
\eqref{eqn:ekorre} that 
\[
\int_{\BS} I \left(B(-r_1e_1,2)
\xleftrightarrow{\ell_{x,p,L}}
B(r_2 e_1,2)\right)\CH(\d p)
\geq \frac{\Gamma(d/2)}{\sqrt{\pi}\Gamma((d+1)/2)}
2^{d-1}\left(\frac{2-\Vert x^0\Vert }{L}\right)^{d-1}.
\]
We can now integrate over $x\in \Cyl_{L/(32\sqrt{d})}$ to obtain
\begin{eqnarray*}
\lefteqn{\mu_\lambda\left((x,p): x\in \Cyl_{L/(32\sqrt{d})}, 
B(-r_1e_1,2)
\xleftrightarrow{\ell_{x,p,L}}
B(r_2 e_1,2)\right)}\\
& & =\lambda \int_{\Cyl_{L/(32\sqrt{d})}} \int_{\BS} 
I \left(B(-r_1e_1,2)
\xleftrightarrow{\ell_{x,p,L}}
B(r_2 e_1,2)\right)\Theta(\d p) \d x\\
& & \geq \lambda \int_{\Cyl_{L/(32\sqrt{d})}} \int_{\BS} 
I \left(B(-r_1e_1,2)
\xleftrightarrow{\ell_{x,p,L}}
B(r_2 e_1,2)\right)\delta \CH(\d p) \d x\\
& & \geq \lambda \delta \frac{2^{d-1}\Gamma(d/2)}{\sqrt{\pi}\Gamma((d+1)/2)}
\int_{\Cyl_{L/(32\sqrt{d})}} 
\left(\frac{2-\Vert x^0\Vert }{L}\right)^{d-1} \d x \\
& & =\lambda \delta 
\frac{2^{d-1}\Gamma(d/2)}{\sqrt{\pi}\Gamma((d+1)/2)} 
\frac{L}{32\sqrt{d}} L^{-d+1}
\int_{\Vert y\Vert \leq 2} (2-\Vert y\Vert )^{d-1} \d y
\end{eqnarray*}
where the last integral is for $y\in \BR^{d-1}$. We therefore have that
\begin{eqnarray*}
\lefteqn{\int_{\Vert y\Vert \leq 2} (2-\Vert y\Vert )^{d-1} \d y}\\
& & =\frac{2 \pi^{(d-1)/2}}{\Gamma((d-1)/2)}
\int_0^2 (2-r)^{d-1} r^{d-2} \d r
=\frac{2 \pi^{(d-1)/2}}{\Gamma((d-1)/2)}
\frac{2^{2(d-1)}\Gamma(d-1) \Gamma(d)}{\Gamma(2d-1)}
\end{eqnarray*}
so that finally 
\begin{eqnarray*}
\lefteqn{\mu_\lambda\left((x,p): x\in \Cyl_{L/(32\sqrt{d})}, 
B(-r_1e_1,2)
\xleftrightarrow{\ell_{x,p,L}}
B(r_2 e_1,2)\right)} \\
& & \geq \lambda \delta 
\frac{L}{32\sqrt{d}} 
L^{-d+1} \frac{2^{d-1}\Gamma(d/2)}{\sqrt{\pi}\Gamma((d+1)/2)}
\frac{2 \pi^{(d-1)/2}}{\Gamma((d-1)/2)}
\frac{2^{2(d-1)}\Gamma(d-1) \Gamma(d)}{\Gamma(2d-1)} \\
& & =\lambda \delta L^{-d+2}
\frac{\Gamma(d)2^{1-d}\sqrt{\pi}}{\Gamma((d+1)/2)}
\frac{\Gamma(d-1)2^{1-(d-1)}\sqrt{\pi}}{\Gamma((d-1)/2)}
\frac{2^{5(d-1)}}{32}\frac{1}{\sqrt{d}}\pi^{d/2-2} 
\frac{\Gamma(d/2)}{\Gamma(2d-1)}\\
& & =\lambda \delta L^{-d+2}
\Gamma(d/2)
\Gamma(d/2)
2^{5(d-2)}\frac{1}{\sqrt{d}}\pi^{d/2-2} 
\frac{\Gamma(d/2)}{\Gamma(2d-1)} \\
& & =\lambda \delta 2^{5(d-2)}\frac{1}{\sqrt{d}}\pi^{d/2-2}
\frac{\Gamma(d/2)^3}{\Gamma(2d-1)} L^{-d+2},
\end{eqnarray*}
by using the identity $\Gamma(z)\Gamma(z+1/2)=2^{1-2z}\sqrt{\pi}\Gamma(2z)$ 
for $z=d/2$ and $z=(d-1)/2$ in the second to last equality.
\end{proof}

Recall the notation $R^{-16}(D^u)$ from Section \ref{sec:upperbounds}.
We are now ready to prove Lemma \ref{lemma:twoballmeasure2}.

\begin{proof} [Proof of Lemma \ref{lemma:twoballmeasure2}]
We will use Lemma \ref{lemma:twoballmeasure} and need to make preparations 
for this. To that end, consider 
\[
\ell_{\gamma,\zeta}:=\{\gamma+t(\zeta-\gamma): 0 \leq t \leq 1\},
\]
i.e.~the line segment between $\gamma$ and $\zeta.$ Then, let 
$z^*\in \ell_{\gamma,\zeta}$ be such that $z^*_1=-L/4$ so that 
$z^*$ is the point on the line $\ell_{\gamma,\zeta}$ whose 
first coordinate is in the middle of the right-hand side 
and the left-hand side of $D^{(-1,0)}.$ Next, we let 
\[
\ell^*:=\left\{z\in \ell_{\gamma,\zeta}: 
\Vert z-z^* \Vert \leq \frac{L}{32 \sqrt{d}}\right\},
\]
and $S^*:=(\ell^*)^{+2}$ (see Figure \ref{fig:StickinBox} for a 
depiction).
\begin{figure} 
\includegraphics[clip=true, trim=10 200 50 100, width=15.5cm,
page=1]{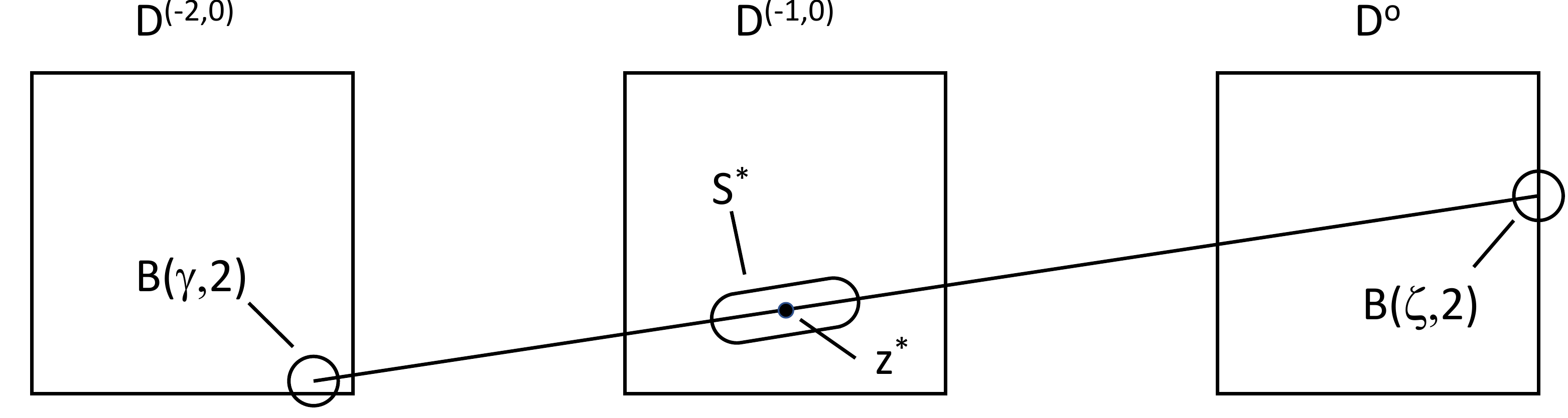} 
\caption{The figure depicts the balls $B(\gamma,2)$ where 
$\gamma \in D^{(-2,0)}$ and $B(\zeta,2)$ where 
$\zeta \in R^{-16}(D^{o}).$ We can also see $z^*\in \ell_{\gamma,\zeta}$ 
and $S^*.$
Note that the region $S^*$ will typically 
not be centered in the middle between $\gamma,\zeta.$
The figure is not drawn to scale.} \label{fig:StickinBox}
\end{figure}
Clearly, the triple $S^*,B(\gamma,2)$ and $B(\zeta,2)$ is just 
a rotated and translated version of the triple 
$S_{L/(32\sqrt{d})}, B(-r_1 e_1,2)$ and $B(r_2 e_1,2)$
of the statement of Lemma \ref{lemma:twoballmeasure}.
In order to use that lemma, we need that $L>32$ which holds by 
assumption, and in addition we need to verify that 
\begin{equation} \label{eqn:maxdist}
\frac{L}{8}< \min(\Vert z^*-\gamma \Vert,\Vert z^*-\zeta \Vert)
\leq \max(\Vert z^*-\gamma \Vert,\Vert z^*-\zeta \Vert)
< \frac{L}{\sqrt{6}}.
\end{equation}
We observe that it follows from the definition of 
$D^u$, i.e.~\eqref{eqn:defDv}, that for any $z\in D^{(-1,0)},$
\begin{equation} \label{eqn:dxgammaupper}
\Vert z-\gamma \Vert^2
\leq \left(\frac{L}{4}+\frac{L}{8 \sqrt{d}}\right)^2
+(d-1)\frac{L^2}{64 d}
=\frac{L^2}{16}+\frac{L^2}{16 \sqrt{d}}+\frac{L^2}{64}
< \frac{L^2}{6},
\end{equation}
for every $d\geq 2.$ Here, we used that $D^{u}$ has side length 
$\frac{L}{8 \sqrt{d}},$ and that
the distance between the centers of two neighboring boxes is $L/4.$
 Furthermore, it is easy to see that 
\begin{equation} \label{eqn:dxgammalower}
\Vert z-\gamma \Vert \geq \frac{L}{4}-\frac{L}{8 \sqrt{d}}
> \frac{L}{8}
\end{equation}
for $d\geq 2.$
Clearly \eqref{eqn:dxgammaupper} and \eqref{eqn:dxgammalower} 
must also hold for $\zeta$ in place of $\gamma$ and so \eqref{eqn:maxdist}
follows since $z^*\in D^{(-1,0)}.$

Next, we need to show that 
\begin{equation}\label{eqn:S*subset}
S^*\subset D^{(-1,0)}
\end{equation}
since then we can conclude that 
\begin{eqnarray} \label{eqn:muS*}
\lefteqn{\mu_\lambda\left((x,p): x\in D^{(-1,0)}, 
B(\gamma,2)\xleftrightarrow{\ell_{x,p,L}} B(\zeta,2)\right)}\\
& & \geq \mu_\lambda\left((x,p): x\in S^*, 
B(\gamma,2)\xleftrightarrow{\ell_{x,p,L}} B(\zeta,2)\right)\nonumber  \\
& & =\lambda\int_{S^*} \int_{\BS} 
I(B(\gamma,2)\xleftrightarrow{\ell_{x,p,L}} B(\zeta,2)) 
\Theta(\d p) \d x 
\geq \lambda \delta  c_d L^{-d+2}, \nonumber
\end{eqnarray}
where we used Lemma \ref{lemma:twoballmeasure} with $c_d$ as in 
that lemma in the last inequality.

In order to verify \eqref{eqn:S*subset}, observe first that for 
any $z=(z_1,\ldots,z_d)\in \ell^*$ we have that 
$|z_1-z_1^* | \leq \frac{L}{32\sqrt{d}}.$ 
Furthermore, by using that $\gamma\in D^{(-2,0)}$ and 
that $\zeta \in R^{-16}(D^o),$ 
it is a straightforward, although tedious, 
trigonometric exercise (which is outlined below) to conclude that
\begin{equation}\label{eqn:trigineq}
|z_k|\leq \frac{L}{16 \sqrt{d}}-4
\end{equation}
for $k=2,\ldots,d,$ and it follows that \eqref{eqn:S*subset} holds.  

In order to understand why \eqref{eqn:trigineq} is correct,
consider the second coordinate $z_2.$ All other coordinates except 
the first will play no 
role in this explanation and we therefore ignore these.
Note that the extreme case in the sense of $z_2$ being as ``low'' 
as possible would be when $\gamma=(\gamma_1,\ldots,\gamma_d)$
and $\zeta=(\zeta_1,\ldots,\zeta_d)$ are such that
$\gamma_1=-\frac{L}{2}+\frac{L}{16\sqrt{d}}, 
\gamma_2=-\frac{L}{16 \sqrt{d}}, \zeta_1=\frac{L}{16 \sqrt{d}}$ and 
$\zeta_2=-\frac{L}{16 \sqrt{d}}+16$ (see Figure \ref{fig:trig}).
\begin{figure} 
\includegraphics[clip=true, trim=10 100 10 10, width=15.5cm,page=2]{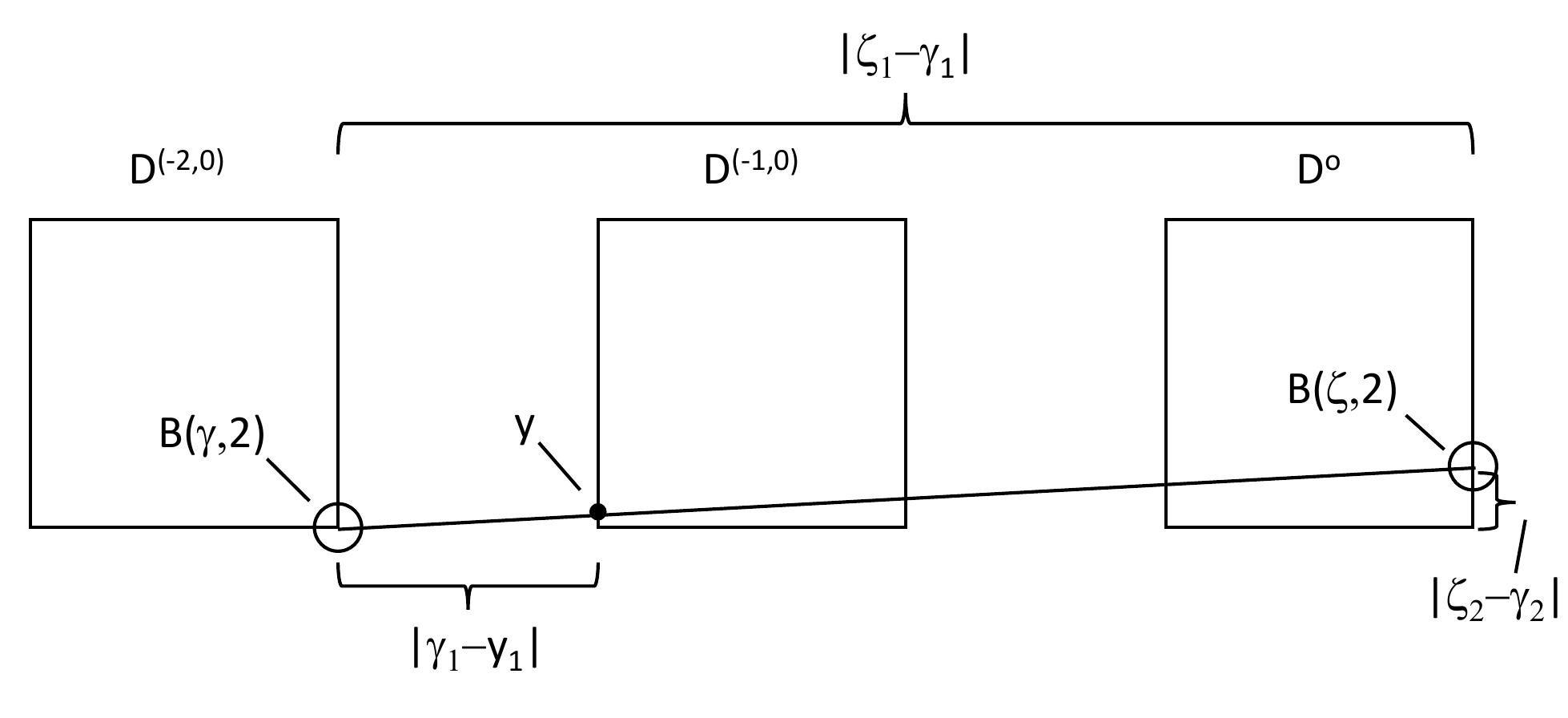} 
\caption{The worst case is depicted, i.e.~ where $\gamma$ is the lower right
corner of $D^{(-2,0)}$ and $\zeta \in R(D^o)$ is as close to the bottom 
of $D^o$ as it is allowed. The figure is not drawn to scale.} \label{fig:trig}
\end{figure}
Then, let $y\in \ell_{\gamma,\zeta}$ be such 
that $y_1=-\frac{L}{4}-\frac{L}{16\sqrt{d}}.$ We see that $y$ is where 
$\ell_{\gamma,\zeta}$ intersects the left side of $D^{(-1,0)}$
and that we must have that
\[
\frac{\left|y_2-\gamma_2\right|}{\left|y_1-\gamma_1\right|}
=\frac{|\zeta_2-\gamma_2|}{|\zeta_1-\gamma_1|}.
\]
We conclude that 
\begin{eqnarray*}
\lefteqn{|y_2-\gamma_2|
=\left|y_2-\left(-\frac{L}{16\sqrt{d}}\right)\right|}\\
& & =\left|y_1-\gamma_1\right|\frac{|\zeta_2-\gamma_2|}{|\zeta_1-\gamma_1|}
=\left(\frac{L}{4}-\frac{L}{8 \sqrt{d}}\right)\frac{16}{L/2}
\geq 4.
\end{eqnarray*}
It follows that for any $z\in \ell^*$ we must have that 
\[
z_2 \geq -\frac{L}{16\sqrt{d}}+4,
\]
and by expanding this argument we conclude that \eqref{eqn:trigineq} must
hold for $k=2,\ldots,d.$ It is worth noting that it is important for this 
argument that $\zeta\in R^{-16}(D^o)$ rather than just being an arbitrary 
point on the right-hand side of $D^o.$ To see this, consider Figure 
\ref{fig:trig} and note that if $\zeta$ would belong to the bottom 
right corner of $D^o,$ then the intersection of $\ell_{\gamma,\zeta}$ 
with $D^{(-1,0)}$ could be along the bottom part of $D^{(-1,0)}.$ Then,  
$S^*$ would spill over the boundary of $D^{(-1,0)}$ and so 
\eqref{eqn:S*subset} would no longer hold.

\end{proof}

\subsection{Conditions for disjointness} \label{app:5.2}

Our first lemma of this subsection is an elementary exercise, and 
in it we obtain the 
smallest distance between a parametrized infinite line and a point on 
another line. For this lemma and the next, let 
\begin{equation} \label{eqn:fdef}
f(t,\tau)=\Vert \ell_{x,p,\infty}(t)-\ell_{y,q,\infty}(\tau)\Vert^2
=\Vert x+tp-\tau q\Vert ^2.
\end{equation}

\begin{lemma}[\bf Distance from a line to a point on another line] \label{lemma:dlptlq}
We have that 
\[
\dist(\ell_{x,p,\infty}(t),\ell_{y,q,\infty})^2
=\Vert x-y\Vert ^2+t^2(1-\langle p,q\rangle^2)-\langle x-y,q\rangle^2
+2t(\langle x-y,p\rangle
-\langle p,q\rangle\langle x-y,q\rangle).
\]
\end{lemma}
\begin{proof}
We start by noting that without loss of generality, we can take $y=o.$
Consider $f(t,\tau)$ from \eqref{eqn:fdef}, then fix $t$ and observe that  
\begin{eqnarray*}
\lefteqn{g(\tau)=f(t,\tau)
=\Vert x\Vert ^2+\Vert tp-\tau q\Vert ^2 +2\langle x,tp-\tau q \rangle}\\
& & =\Vert x\Vert ^2+t^2+\tau^2-2t \tau \langle p,q\rangle
+2t\langle x,p\rangle-2 \tau \langle x,q\rangle,
\end{eqnarray*}
where we used that $\Vert p\Vert =\Vert q\Vert =1$. We see that 
$g'(\tau)=2\tau-2t \langle p,q\rangle-2\langle x,q\rangle,$ and setting 
 $g'(\tau)$ to zero 
we obtain 
$
\tau=t \langle p,q\rangle+\langle x,q\rangle.
$
This clearly corresponds to a minimum, and
inserting this into the expression for $g(\tau)$ we then obtain
\begin{eqnarray*}
\lefteqn{g(t \langle p,q\rangle+\langle x,q\rangle)-\Vert x\Vert ^2}\\
& & =t^2+(t \langle p,q\rangle+\langle x,q\rangle)^2
-2t (t \langle p,q\rangle+\langle x,q\rangle) \langle p,q\rangle \\
& & \hspace{4mm}+2t\langle x,p\rangle
-2 (t\langle p,q\rangle+\langle x,q\rangle) \langle x,q\rangle \\
& & =t^2(1-\langle p,q\rangle^2)-\langle x,q\rangle^2
+2t(\langle x,p\rangle
-\langle p,q\rangle\langle x,q\rangle).
\end{eqnarray*}
The result for general $y$ follows by translation.
\end{proof}

Next, let 
\[
h(t)
=\dist(\ell_{x,p,\infty}(t),\ell_{y,q,\infty})^2=\min_{\tau} f(t,\tau).
\]

\begin{lemma}[\bf Minimizer of distance between two lines] \label{lemma:htminplusa}
If $\langle p,q\rangle\neq 0,$ then $h(t)$ is minimized by 
\[
t_{\min}
=-\frac{\langle x-y,p\rangle- \langle p,q\rangle \langle x-y,q\rangle}
{1-\langle p,q\rangle^2}
\]
and furthermore, for any $a\in \BR,$  
\[
h(t_{\min}+a)=h(t_{\min})+a^2(1-\langle p,q\rangle^2).
\]
\end{lemma}
\noindent
\begin{proof}
By Lemma \ref{lemma:dlptlq} we have that 
\[
h'(t)=2t(1-\langle p,q\rangle^2)
+2(\langle x-y,p\rangle
-\langle p,q\rangle\langle x-y,q\rangle),
\]
and since $h''(t)>0$ if $\langle p,q\rangle\neq 0,$ the function 
$h(t)$ is minimized when $h'(t)=0,$ from which the first statement 
follows. We then see that
\begin{eqnarray*}
\lefteqn{h(t_{\min}+a)=\Vert x-y\Vert ^2}\\
& & \hspace{5mm}+(t_{\min}+a)^2(1-\langle p,q\rangle^2)
-\langle x-y,q\rangle^2+2(t_{\min}+a)(\langle x-y,p\rangle
-\langle p,q\rangle\langle x-y,q\rangle)\\
& & =h(t_{\min})+a^2(1-\langle p,q\rangle^2)
+2at_{\min}(1-\langle p,q\rangle^2)
+2a(\langle x-y,p\rangle-\langle p,q\rangle\langle x-y,q\rangle)\\
& & =h(t_{\min})+a^2(1-\langle p,q\rangle^2),
\end{eqnarray*}
by using the expression for $t_{\min}.$
\end{proof}

\begin{proof}[Proof of Lemma \ref{lemma:t1tau1}]
Let $t_{\min}$ be as in in the statement of Lemma \ref{lemma:htminplusa} 
and observe that by that same lemma,
\[
h(t_{\min}+a)=h(t_{\min})+a^2(1-\langle p,q \rangle^2)
\geq a^2\frac{1}{2},
\]
by our assumption on $|\langle p,q \rangle|.$ Let $t_1,\tau_1$
be such that \eqref{eqn:dsmallerthanrho} holds 
(i.e.~$\Vert \ell_{x,p,\infty}(t_1)-\ell_{y,q,\infty}(\tau_1) \Vert
\leq 2$). We can then conclude that 
$t_1\in(t_{\min}-2\sqrt{2},t_{\min}+2\sqrt{2}),$
since for any $t\not \in(t_{\min}-2\sqrt{2},t_{\min}+2\sqrt{2}),$
we have that $h(t)\geq \frac{(2\sqrt{2})^2}{2}=4$ by the observation above. 
Furthermore we see that for any 
$t\not \in (t_{\min}-6\sqrt{2},t_{\min}+6\sqrt{2})$
\[
h(t)\geq (6 \sqrt{2})^2 \frac{1}{2}=36.
\]
It follows that if $|t-t_1|\geq 6\sqrt{2}+2\sqrt{2}=8\sqrt{2},$ 
then $h(t)\geq 36$ so that for any such $t,$
\[
\Vert \ell_{x,p,\infty}(t)-\ell_{y,q,\infty}(\tau)\Vert
\geq 6,
\]
for every $\tau.$ Finally, we simply observe that $8\sqrt{2}\leq 12.$
\end{proof}

{\bf Acknowledgement:} The author would like to thank the 
two anonymous referees for many valuable comments and suggestions. The 
author would also like to thank Ronald Meester for comments on an earlier
version of the paper.

\end{appendices}

\end{document}